\documentclass[12pt, reqno, a4paper]{amsart}
\usepackage{amsmath, amssymb, amsthm, amscd}
\usepackage[T2A, T1]{fontenc}
\usepackage{txfonts}	
\usepackage{eucal}
\usepackage[dvips]{color}
\usepackage{multicol}
\usepackage[all]{xy}		
\usepackage{graphicx}
\usepackage{color}
\usepackage{colordvi}
\usepackage{xspace}
\usepackage{tikz}
\usepackage{enumitem}
\usepackage{ulem}
\usepackage[colorlinks,final,backref=page,hyperindex]{hyperref}
\allowdisplaybreaks % page break for formulars

\newif\ifflabel\flabelfalse
\ifflabel
\else
	
\fi
\newtheorem{theorem}{Theorem}[section]
\newtheorem{lemma}[theorem]{Lemma}
\newtheorem{corollary}[theorem]{Corollary}
\newtheorem{proposition}[theorem]{Proposition}
\theoremstyle{definition}
\newtheorem{definition}[theorem]{Definition}
\newtheorem{example}[theorem]{Example}
\newtheorem{remark}[theorem]{Remark}

\newcommand{\Char}{\mathrm{char}}
\newcommand{\End}{\mathrm{End}}
\newcommand{\Hom}{\mathrm{Hom}}
\newcommand{\id}{\mathrm{id}}

\begin{document}

\title{Quasi-triangular Jordan D-bialgebras and extended relative Rota-Baxter operators}

%%%%%%%%%%%%%%%%%%%%%%%%%%%%%%%%%%%%%%%%%%%%%%%%%%%%%%%%%%%%%%%%%%%%%%%%%%%%%%%%
% Author Informations

\author[D.~Lu]{Dilei Lu}
\address{College of Applied Science, Beijing Information Science and Technology University, Beijing 100192, China}
\email{ludyray@bistu.edu.cn}

\author[D.~P.~Hou]{Dongping Hou}
\address{School of Mathematics, Yunnan Normal University, Kunming 650500, China}
\email{houdpmath@163.com}

\author[Y.~Lin]{Yuanchang Lin$^{\ast}$}
\address{School of Mathematics, North University of China, Taiyuan 030051, China}
\email{linyuanchang@mail.nankai.edu.cn}

\thanks{All authors contributed equally to this work. The author order deviates from the alphabetical convention solely to comply with the first author's institutional evaluation policy, which requires first-authorship for official recognition.}
\thanks{$^\ast$Corresponding author.}
%%%%%%%%%%%%%%%%%%%%%%%%%%%%%%%%%%%%%%%%%%%%%%%%%%%%%%%%%%%%%%%%%%%%%%%%%%%%%%%%



\date{\today}

\begin{abstract}
	This paper introduces quasi-triangular Jordan D-bialgebras, which are constructed from solutions of the Jordan Yang-Baxter equation (JYBE) with invariant symmetric parts.
	We first develop a systematic operator approach to the JYBE by introducing the notion of extended relative Rota-Baxter operators on Jordan algebras.
	Such operators with their extensions are shown to induce new Jordan algebra structures. 
	Moreover, the symmetrizer-antisymmetrizer decomposition of linear maps reduces the study of extended relative Rota-Baxter operators to both pairs of homomorphisms of Jordan algebras and relative Rota-Baxter operators, respectively.
	This operator framework subsequently yields a characterization of solutions of the JYBE whose symmetric parts are invariant.
	These operator forms are further investigated in the context of quadratic Jordan algebras and semi-direct product Jordan algebras, respectively, leading to explicit constructions of solutions of the JYBE.
	A factorizable Jordan D-bialgebra is then introduced as a special class of quasi-triangular ones, which naturally factorizes the underlying Jordan algebra.
	We prove that the Drinfeld classical double of any Jordan D-bialgebra admits a factorizable Jordan D-bialgebra structure.
	Finally, the operator perspective gives rise to the notion of quadratic Rota-Baxter Jordan algebras: those of zero weight yield triangular Jordan D-bialgebras, while those of nonzero weight are shown to be in one-to-one correspondence with factorizable Jordan D-bialgebras.
\end{abstract}

\maketitle

\tableofcontents

%%%%%%%%%%%%%%%%%%%%%%%%%%%%%%%%%%%%%%%%%%%%%%%%%%%%%%%%%%%%%%%%%%%%%%%%%%%%%%%%
%%%%%%%%%%%%%%%%%%%%%%%%%%%%%%%%%%%%%%%%%%%%%%%%%%%%%%%%%%%%%%%%%%%%%%%%%%%%%%%%
%%%%%%%%%%%%%%%%%%%%%%%%%%%%%%%%%%%%%%%%%%%%%%%%%%%%%%%%%%%%%%%%%%%%%%%%%%%%%%%%
\section{Introduction}
Jordan algebras are an important class of commutative non-associative algebras, introduced by the physicist P.~Jordan in connection with quantum mechanics~\cite{jordan1932klasse, jordan1934algebraic}.
They can be viewed as the ``dual'' counterpart to Lie algebras, in the sense that the commutator and anti-commutator of an associative algebra give rise to a Lie algebra and a Jordan algebra, respectively.
It should be noted, however, that not every Jordan algebra arises in this manner: those that do are called \emph{special} Jordan algebras, while those that do not are termed \emph{exceptional}~\cite{albert1947structure, jacobson1968structure}.
Subsequently, Jordan algebras have emerged as a fundamental algebraic structure that appears in various areas of modern mathematics, including differential geometry~\cite{chu2006grassmann, chu2008jordan,  iorduanescu2000jordan, kocher1970jordan}, Lie and representation theory~\cite{bertram2002geometry, koecher1967imbedding, kostant1993jordan, meyberg1970jordan}, complex and harmonic analysis~\cite{kaup2000symmetric, resnikoff1975theta, upmeier1986jordan, upmeier2011symmetric}, and operad theory~\cite{ bagherzadeh2017trialgebras, bai2013splitting, kashuba2021free}.   

The theory of Jordan D-bialgebras was initiated by Zhelyabin~\cite{zhelyabin1997jordan,zhelyabin2000class}.
A new approach was recently developed in~\cite{hou2022new}, where the equivalence between Jordan D-bialgebras and double constructions of pseudo-Euclidean Jordan algebras was established. 
Meanwhile, in the theory of Lie bialgebras, quasi-triangular structures~\cite{drinfeld1986quantum} have played a central role in mathematical physics. 
Factorizable Lie bialgebras, as a special subclass, are of particular importance: they connect classical $r$-matrices with factorization problems and are essential to the study of integrable systems~\cite{bai2010nonabelian, reshetikhin1988quantum,semenov2003integrable}.
Recently, the quasi-triangular and factorizable theories have been extended to antisymmetric infinitesimal bialgebras~\cite{sheng2023quasi}, Poisson bialgebras~\cite{lin2026quasitriangular}, pre-Lie bialgebras~\cite{wang2024quasi} and Leibniz bialgebras~\cite{bai2024quasi}.
Note that the notions of triangular and quasi-triangular Jordan D-bialgebras were first introduced by Zhelyabin~\cite{zhelyabin1999finite}, building on the corresponding Lie bialgebra concepts via the Kantor-Koecher-Tits (KKT) construction. 
However, it remains an open problem whether the KKT construction can be applied to arbitrary Jordan D-bialgebras~\cite{zhelyabin1997jordan, zhelyabin1998jordan}. 
It is therefore desirable to develop a more direct approach to the quasi-triangular and factorizable theories for Jordan D-bialgebras, bypassing the KKT construction.

The purpose of this paper is to develop such theories for Jordan D-bialgebras, in parallel with the Lie bialgebra case, without relying on the KKT construction.
The key technical challenge is that, unlike the cases of Lie bialgebras, where the Yang-Baxter-type equations involve 3-tensors and the compatibility identities involve 3 elements, the Jordan D-bialgebra compatibility identities involve 4 elements.
This additional complexity presents a significant difficulty.
To overcome it, we first simplify the condition for a Jordan D-bialgebra to be coboundary, showing that the symmetric part of the 2-tensor $r \in A \otimes A$ must be invariant and that $r$ should satisfy a 3-tensor Yang-Baxter-type equations, namely the Jordan Yang-Baxter equation (JYBE).
We show that a solution of the JYBE with invariant symmetric part gives rise to a \emph{quasi-triangular} Jordan D-bialgebra, while a skew-symmetric solution yields a \emph{triangular} one.

To characterize such solutions, we adopt an operator approach that is both conceptually unified and computationally powerful.
We begin by introducing extended relative Rota-Baxter operators on Jordan algebras, which generalize the ordinary ones (also known as $\mathcal{O}$-operators) by incorporating an extension term.
The power of this approach is demonstrated via the symmetrizer-antisymmetrizer decomposition, which reduces the study of extended relative Rota-Baxter operators to both pairs of homomorphisms of Jordan algebras and relative Rota-Baxter operators respectively, and uncovers the intrinsic structure of the framework.
Moreover, via the symmetrizer-antisymmetrizer decomposition, the skew-symmetric and symmetric parts of solutions of the JYBE are shown to correspond to an extended relative Rota-Baxter operator together with an extension, respectively.

We then obtain a characterization of such solutions in terms of extended relative Rota-Baxter operators. 
The operator framework culminates in a description of the solutions of the JYBE with invariant symmetric part in terms of both pairs of homomorphisms of Jordan algebras and relative Rota-Baxter operators, respectively.
In the skew-symmetric case, the extended relative Rota-Baxter operator reduces to the ordinary relative Rota-Baxter operator of weight $0$, recovering the classical correspondence between $\mathcal{O}$-operators and the Yang-Baxter equation. Then we further demonstrate the versatility of the operator approach by investigating the JYBE on quadratic Jordan algebras and semi-direct product Jordan algebras.
These results provide systematic methods for constructing solutions of the JYBE from extended relative Rota-Baxter operators, and vice versa.

We also introduce \emph{factorizable} Jordan D-bialgebras as a special class of quasi-triangular ones.
We demonstrate that every factorizable Jordan D-bialgebra induces a factorization of the underlying Jordan algebra, and the Drinfeld classical double of an arbitrary Jordan D-bialgebra admits a natural factorizable Jordan D-bialgebra structure. Finally, building upon the operator perspective, we introduce the notion of a quadratic Rota-Baxter Jordan algebra, which consists of a quadratic Jordan algebra together with a Rota-Baxter operator of weight $\lambda$ satisfying a compatibility condition.
We show that a quadratic Rota-Baxter Jordan algebra of zero weight yields a triangular Jordan D-bialgebra, while a one-to-one correspondence is established between factorizable Jordan D-bialgebras and quadratic Rota-Baxter Jordan algebras of nonzero weight.

The paper is organized as follows.
In Section~\ref{sec:ajaerrb}, we recall the definitions and basic properties of Jordan algebras and their representations, and introduce $A$-module Jordan algebras and extended relative Rota-Baxter operators.
We then establish the symmetrizer-antisymmetrizer decomposition of linear maps, showing that extended relative Rota-Baxter operators are equivalent to both pairs of relative Rota-Baxter operators and homomorphisms of Jordan algebras, respectively.
Section~\ref{sec:JYBE} focuses on quasi-triangular Jordan D-bialgebras. We recall the definition of Jordan D-bialgebras, simplify the coboundary condition, derive the operator forms of the JYBE via extended relative Rota-Baxter operators, and investigate these structures on quadratic Jordan algebras and semi-direct product Jordan algebras.
In Section~\ref{sec:quadrbj}, we introduce factorizable Jordan D-bialgebras and quadratic Rota-Baxter Jordan algebras, prove that the Drinfeld double of any Jordan D-bialgebra admits a factorizable Jordan D-bialgebra structure, and establish a bijective correspondence between factorizable Jordan D-bialgebras and quadratic Rota-Baxter Jordan algebras of nonzero weight.

Throughout this paper, we work over a base field $\mathbf{k}$ of characteristic $0$, and all vector spaces and algebras are assumed to be finite-dimensional.
We adopt the following conventions and notations.
\begin{enumerate}
	\item[(i)]
	      Let $(A, \diamond)$ be a vector space equipped with a binary operation $\diamond: A \otimes A \to A$.
	      Let $L_\diamond(x)$ and $R_\diamond(x)$ denote the left and right multiplication operators, that is
	      \begin{equation*}
		      L_\diamond(x)y = R_\diamond(y)x = x \diamond y, \;\; \forall x, y \in A.
	      \end{equation*}

	\item[(ii)]
	      Let $(A, \diamond)$ be a vector space equipped with a binary operation $\diamond: A \otimes A \to A$.
	      Let $r = \sum_{i} x_i \otimes y_i \in A \otimes A$.
	      Set
	      \begin{align*}
		      r_{12} = \sum\nolimits_{i} x_i \otimes y_i \otimes 1, \;\;
		      r_{13} = \sum\nolimits_{i} x_i \otimes 1 \otimes y_i, \;\;
		      r_{23} = \sum\nolimits_{i} 1 \otimes x_i \otimes y_i, \\
		      r_{21} = \sum\nolimits_{i} y_i \otimes x_i \otimes 1, \;\;
		      r_{31} = \sum\nolimits_{i} y_i \otimes 1 \otimes x_i, \;\;
		      r_{32} = \sum\nolimits_{i} 1 \otimes y_i \otimes x_i,
	      \end{align*}
	      where $1$ is the unit if $(A, \diamond)$ is unital or a symbol playing a similar role as the unit for the non-unital cases.
	      Furthermore, define compound symbols such as $r_{12} \diamond r_{13}$ by
	      \begin{equation*}
		      r_{12} \diamond r_{13} = \sum\nolimits_{i, j} x_i \diamond x_j \otimes y_i \otimes y_j.
	      \end{equation*}

	\item[(iii)]
	      Let $V, W$ be two vector spaces and $T: V \to W$ be a linear map.
	      Denote the dual map by $T^*: W^* \to V^*$, which is defined by
	      \begin{equation*}
		      \langle T^*(\xi^*), v \rangle = \langle \xi^*, T(v) \rangle, \;\; \forall v \in V, \; \xi^* \in W^*,
	      \end{equation*}
		  where $\langle \ ,\ \rangle$ is the standard pairing between the dual space $V^*$ and $V$.

	\item[(iv)]
	      Let $A, V$ be vector spaces.
	      For a linear map $\zeta: A \to \End_{\mathbf{k}}(V)$, define a linear map $\zeta^*: A \to \End_{\mathbf{k}}(V^*)$ by $\zeta^*(a) = (\zeta(a))^*$, or more explicitly,
	      \begin{equation*}
		      \langle \zeta^*(a)v^*, u\rangle = \langle v^*, \zeta(a) u\rangle, \;\; \forall a \in A, \; u \in V, \; v^* \in V^*.
	      \end{equation*}

\end{enumerate}

%%%%%%%%%%%%%%%%%%%%%%%%%%%%%%%%%%%%%%%%%%%%%%%%%%%%%%%%%%%%%%%%%%%%%%%%%%%%%%%%
%%%%%%%%%%%%%%%%%%%%%%%%%%%%%%%%%%%%%%%%%%%%%%%%%%%%%%%%%%%%%%%%%%%%%%%%%%%%%%%%
%%%%%%%%%%%%%%%%%%%%%%%%%%%%%%%%%%%%%%%%%%%%%%%%%%%%%%%%%%%%%%%%%%%%%%%%%%%%%%%%
\section{\texorpdfstring{$A$}{A}-module Jordan algebras and extended relative Rota-Baxter operators}\label{sec:ajaerrb}

In this section, we first recall the fundamental properties and notions of Jordan algebras.
We then introduce $A$-module Jordan algebras, relative Rota-Baxter operators and extended relative Rota-Baxter operators.
We show that extended relative Rota-Baxter operators induce new Jordan algebra structures and provide equivalent characterizations of extended relative Rota-Baxter operators via the symmetrizer-antisymmetrizer decomposition, which lays the necessary foundations for the remainder of this paper.

%%%%%%%%%%%%%%%%%%%%%%%%%%%%%%%%%%%%%%%%%%%%%%%%%%%%%%%%%%%%%%%%%%%%%%%%%%%%%%%%
\subsection{\texorpdfstring{$A$}{A}-module Jordan algebras and relative Rota-Baxter operators}

\begin{definition}
	A {\bf Jordan algebra} is a pair $(A, \circ)$, where $A$ is a vector space and $\circ : A \otimes A \to A$ is a commutative binary operation satisfying the following \textbf{Jordan identity}:
	\begin{equation}
		((x \circ x) \circ y) \circ x = (x \circ x) \circ (y \circ x), \;\; \forall x, y \in A. \label{eq:Jidentity}
	\end{equation}
	A Jordan algebra $(A, \circ)$ is called \textbf{trivial} if $x \circ y = 0$ for all $x, y \in A$.
	A linear map $f: A \to B$ between Jordan algebras $(A, \circ_A)$ and $(B, \circ_B)$ is called a \textbf{homomorphism} of Jordan algebras if $f(x \circ_A y) = f(x) \circ_B f(y)$ for all $x, y \in A$.
\end{definition}

\begin{remark}
	When $\Char(\mathbb{\mathbf{k}}) \neq 2, 3$, it was shown in~\cite{albert1947structure} that the Jordan identity~\eqref{eq:Jidentity} is equivalent to the following identity for all $x, y, z, w \in A$
	\begin{align}
		& ((x \circ y) \circ w) \circ z + ((y \circ z) \circ w) \circ x+((z \circ x) \circ w) \circ y \notag            \\
		=\; & (x \circ y) \circ(w \circ z) + (y \circ z) \circ (w \circ x) + (z \circ x) \circ (w \circ y). \label{eq:jord}
	\end{align}
\end{remark}

\begin{lemma}
	Let $(A, \circ)$ be a Jordan algebra.
	Then the following equality holds for all $x, y, z, w \in A$
	\begin{align}
		& x \circ ( y\circ (z \circ w))-y \circ ( x\circ (z \circ w)) \notag \\
		=\; & z \circ ( x\circ (w \circ y)) +w \circ ( x\circ (y \circ z))- z \circ ( y\circ (w \circ x)) -w \circ ( y \circ (x \circ z)), \label{ID:jord2}
	\end{align}
\end{lemma}
\begin{proof}
	For all $x, y, z, w \in A$, we have
	\begin{align*}
		& x \circ ( y\circ (z \circ w)) - y \circ ( x\circ (z \circ w)) \\
		\overset{\eqref{eq:jord}}{=}\; & (x \circ y) \circ (z \circ w) + (z \circ y) \circ (w \circ x) + (w \circ y) \circ (x \circ z) - z \circ ( y\circ (w \circ x)) - w \circ ( y\circ (x \circ z)) \\
		& - (y \circ x) \circ (z \circ w) - (z \circ x) \circ (w \circ y) - (w \circ x) \circ (y \circ z) + z \circ ( x\circ (w \circ y)) + w \circ ( x\circ (y \circ z)) \\
		\overset{\hphantom{\eqref{eq:jord}}}{=}\; & z \circ ( x\circ (w \circ y)) +w \circ ( x\circ (y \circ z)) - z \circ ( y\circ (w \circ x)) -w \circ ( y\circ (x \circ z)).
	\end{align*}
	The proof is complete.
\end{proof}

\begin{example}\label{ex:JordanXY}
	The following examples of Jordan algebras are taken from~\cite{kashuba2016variety}.
	\begin{enumerate}
		\item[(a)]
		    $(A, \circ)$ is a 2-dimensional Jordan algebra with a basis $\{e_{1}, e_{2}\}$ whose product of $\circ$ is given by
		    \begin{equation*}
			    e_{1} \circ e_{1} = e_{1}, \quad
			    e_{1} \circ e_{2} = \frac{1}{2}e_{2} = e_{2} \circ e_{1}, \quad
			    e_{2} \circ e_{2} = 0. 
		    \end{equation*}

		\item[(b)]
		    $(A, \circ)$ is a 3-dimensional Jordan algebra with a basis $\{e_{1}, e_{2}, e_{3}\}$ whose non-zero product of $\circ$ is given by
		    \begin{equation*}
			    e_{2} \circ e_{2} = e_{1}, \quad
			    e_{3} \circ e_{3} = -e_{1}, \quad
			    e_{1} \circ e_{i} = e_{i} = e_{i} \circ e_{1}, \quad i=1,2,3. 
		    \end{equation*}
	\end{enumerate}
\end{example}

\begin{definition}
	A \textbf{representation} of a Jordan algebra $(A, \circ)$ is a pair $(V; \rho)$,
	where $V$ is a vector space and $\rho: A \to \mathfrak{gl}(V)$ is a linear map such that for all $x, y, z \in A$,
	\begin{align}
		 & \rho(x) \rho(y) \rho(z) + \rho(z) \rho(y) \rho(x) + \rho((x \circ z) \circ y) = \rho(x) \rho(y \circ z) + \rho(y) \rho(z \circ x) + \rho(z) \rho(x \circ y), \label{repJor1} \\
		 & [\rho(x), \rho(y \circ z)] + [\rho(y), \rho(z \circ x)] + [\rho(z), \rho(x \circ y)] = 0. \label{repJor2}
	\end{align}
	A linear map $\varphi: V_{1} \to V_{2}$ between representations $(V_{1}; \rho_{1})$ and $(V_{2}; \rho_{2})$ is called a \textbf{homomorphism} of representations if $\varphi \rho_{1}(x) = \rho_{2}(x) \varphi$ for all $x \in A$;
	if in addition, $\varphi$ is a linear isomorphism, then it is called an \textbf{isomorphism} of representations.
\end{definition}

Note that if $(V; \rho)$ is a representation of $(A, \circ)$, then so is $(V^*; \rho^*)$ \cite{jacobson1951general}.

\begin{example}
	Let $(A, \circ)$ be a Jordan algebra.
	Then both $(A; L_\circ)$ and $(A^*; L_{\circ}^*)$ are representations of $(A, \circ)$, called the {\bf regular representation} and the \textbf{coregular representation}, respectively.
\end{example}

\begin{definition}
	Let $(A, \circ)$ and $(V, \circ_V)$ be Jordan algebras, and $\rho: A \to \End_{\mathbf{k}}(V)$ be a linear map.
	The pair $((V, \circ_V); \rho)$ is called an \textbf{$A$-module Jordan algebra} if $(V; \rho)$ is a representation of $(A, \circ)$ and the following equalities hold for all $x, y \in A$ and $u, v, w \in V$:
	\begin{align}
		& u \circ_V \rho(x)(v \circ_V w) + v \circ_V \rho(x)(w \circ_V u) + w \circ_V \rho(x)(u \circ_V v) \nonumber \\
		=\; & (\rho(x)u) \circ_V (v \circ_V w) + (\rho(x)v) \circ_V (w \circ_V u) + (\rho(x)w) \circ_V (u \circ_V v), \label{Ajord1} \\
		& \rho(x)\big((u \circ_V v) \circ_V w\big) + v \circ_V \big(w \circ_V \rho(x)u\big) + u \circ_V \big(w \circ_V (\rho(x)v)\big) \nonumber \\
		=\; & (u \circ_V v) \circ_V (\rho(x)w) + (v \circ_V w) \circ_V (\rho(x)u) + (w \circ_V u) \circ_V (\rho(x)v),\label{Ajord2} \\
		& \rho(y)\rho(x)(u \circ_V v) + (\rho(x)\rho(y)v) \circ_V u + (\rho(x)\rho(y)u) \circ_V v \nonumber \\
		=\; & \rho(x \circ y)(u \circ_V v) + (\rho(x)u) \circ_V (\rho(y)v) + (\rho(y)u) \circ_V (\rho(x)v), \label{Ajord3} \\
		& \big(\rho(x \circ y)u\big) \circ_V v + \rho(x)\big(u \circ_V (\rho(y)v)\big) + \rho(y)\big(u \circ_V (\rho(x)v)\big) \nonumber \\
		=\; & \rho(x \circ y)(u \circ_V v) + (\rho(x)u) \circ_V (\rho(y)v) + (\rho(y)u) \circ_V (\rho(x)v). \label{Ajord4}
	\end{align}
\end{definition}

\begin{proposition}
	Let $(A, \circ)$ be a Jordan algebra, $V$ be a vector space with a binary operation $\circ_V$, and $\rho: A \to \End_{\mathbf{k}}(V)$ be a linear map.
	Define a binary operation $\bullet$ on $A \oplus V$ by
	\begin{equation*}
		(x + u) \bullet (y + v) := x \circ y + \rho(x) v + \rho(y) u + u \circ_V v, \;\; \forall x, y \in A, \; u, v \in V.
	\end{equation*}
	Then $((V, \circ_V); \rho)$ is an $A$-module Jordan algebra if and only if $(A \oplus V, \bullet)$ is a Jordan algebra, which is called the \textbf{semi-direct product of $(A, \circ)$ and $((V, \circ_V); \rho)$}, and denoted by $(A \ltimes_{\rho} V, \bullet)$.
\end{proposition}
\begin{proof}
	The proof follows by a similar argument as in~\cite[Theorem 2.10]{hou2022new}.
\end{proof}

\begin{example}
	Let $(A, \circ)$ be a Jordan algebra. 
	Then
	\begin{enumerate}
		\item[(i)]
		    $((A, \circ); L_{\circ})$ is an $A$-module Jordan algebra.

		\item[(ii)]
		    A representation of $(A, \circ)$ is equivalent to an $A$-module Jordan algebra equipped with the trivial Jordan algebra structure.
		    In the sequel, we will adopt this perspective for convenience and always regard a representation of $(A, \circ)$ as such an $A$-module Jordan algebra.
	\end{enumerate}
\end{example}

We now introduce the notion of (relative) Rota-Baxter operators of weights associated to $A$-module Jordan algebras.
These operators are motivated by their counterparts in the context of Lie and associative algebras, where they are also known as $\mathcal{O}$-operators~\cite{bai2010double, bai2010nonabelian, bai2012O, kupershmidt1999classical}.

\begin{definition}\label{def:oo}
	Let $(A, \circ)$ be a Jordan algebra, $((V, \circ_V); \rho)$ be an $A$-module Jordan algebra and $\lambda \in \mathbf{k}$.
	A linear map $T: V \to A$ is called a \textbf{relative Rota-Baxter operator} of weight $\lambda$ on $(A, \circ)$ associated to $((V, \circ_V); \rho)$ if
	\begin{equation*}
		T(u) \circ T(v) = T\big( \rho(T(u))v + \rho(T(v))u + \lambda u \circ_V v \big), \;\; \forall u, v \in V.
	\end{equation*}
	A relative Rota-Baxter operator of weight $\lambda$ on $(A, \circ)$ associated to $((A, \circ); L_{\circ})$ is called a {\bf Rota-Baxter operator} of weight $\lambda$ on $(A, \circ)$, i.e.,
	\begin{equation*}
		T(x) \circ T(y) = T\big( T(x) \circ y + x \circ T(y) + \lambda x \circ y \big), \;\; \forall x, y \in A.
	\end{equation*}
\end{definition}

\begin{remark}
	$T$ is a relative Rota-Baxter operator of weight $\lambda$ on $(A, \circ)$ associated to $((V, \circ_V); \rho)$ if and only if $T$ is a relative Rota-Baxter operator of weight $1$ on $(A, \circ)$ associated to $((V, \lambda \circ_V); \rho)$.
	If, in addition $\lambda \neq 0$, this is equivalent to $\frac{T}{\lambda}$ being a relative Rota-Baxter operator of weight $1$ on $(A, \circ)$ associated to $((V, \circ_V); \rho)$.
\end{remark}

We conclude this subsection by establishing the relationship between relative Rota-Baxter operators and Rota-Baxter operators.

\begin{proposition}
	Let $(A, \circ)$ be a Jordan algebra, $((V, \circ_V); \rho)$ be an $A$-module Jordan algebra, and $(A \ltimes_{\rho}V, \bullet)$ be the semi-direct product of $(A, \circ)$ and $((V, \circ_V); \rho)$.
	Let $T: V \to A$ be a linear map, and $\lambda \in \mathbf{k}$.
	Then the following conditions are equivalent.
	\begin{enumerate}[label=(\roman*)]
		\item \label{xyzw:semi:1}
		      $T$ is a relative Rota-Baxter operator of weight $\lambda$ on $(A, \circ)$ associated to $((V, \circ_V); \rho)$.

		\item \label{xyzw:semi:2}
		      The linear map $\widehat{T}$ defined by
		      \begin{align}
			      \widehat{T}: A \oplus V \to A \oplus V, \;\; x + u \mapsto T(u) - \lambda x, \;\; \forall x \in A,\; u \in V,
		      \end{align}
		      is a Rota-Baxter operator of weight $\lambda$ on the Jordan algebra $(A \ltimes_{\rho}V, \bullet)$.

		\item \label{xyzw:semi:3}
		      The linear map $-\lambda \id - \widehat{T}$ defined by
		      \begin{align}
			      -\lambda \id - \widehat{T}: A \oplus V \to A \oplus V , \;\; x + u \mapsto -  T(u)-\lambda u, \quad \forall x \in A,\; u \in V,
		      \end{align}
		      is a Rota-Baxter operator of weight $\lambda$ on the Jordan algebra $(A \ltimes_{\rho}V, \bullet)$.
	\end{enumerate}
\end{proposition}
\begin{proof}
	The implication \ref{xyzw:semi:2} $\Longleftrightarrow $ \ref{xyzw:semi:3} is a consequence of the standard fact that $P: A \to A$ is a Rota-Baxter operator of weight $\lambda$ on $(A, \circ)$ if and only if $-\lambda \id-P: A \to A$ is a Rota-Baxter operator of weight $\lambda$.
	Thus, it is sufficient to establish~\ref{xyzw:semi:1} $\Longleftrightarrow$ \ref{xyzw:semi:2}.
	For all $x, y \in A$ and $u, v \in V$, we have
	\begin{align*}
		\widehat{T}(x + u) \bullet \widehat{T}(y+ v)  & =  T(u) \circ T(v) -  \lambda T(u) \circ y -  \lambda x \circ T(v)+\lambda^{2} x \circ y,   \\
		\widehat{T}(\widehat{T}(x + u) \bullet (y+v)) & = \lambda^{2} x \circ y - \lambda T(u) \circ y - \lambda T(\rho(x) v) +  T(\rho(T(u)) v), \\
		\widehat{T}((x+u) \bullet \widehat{T}(y+v))   & =\lambda^{2} x \circ y -  \lambda x \circ T(v) -  \lambda T( \rho(y)u) +  T( \rho(T(v))u),  \\
		\widehat{T}((x+ u) \bullet (y+v))         & = -\lambda x \circ y + T(\rho(x) v) + T(\rho(y)u) + T(u \circ_V v).
	\end{align*}
	Therefore
	\begin{align*}
		& \widehat{T}(x + u) \bullet \widehat{T}(y+ v) - \widehat{T}(\widehat{T}(x + u) \bullet (y+v) + (x+u) \bullet \widehat{T}(y+v) + \lambda (x+ u) \bullet (y+v)) \\
		=\; & T(u) \circ T(v) - T(\rho(T(u)) v + \rho(T(v))u + \lambda u \circ_V v ).
	\end{align*}
	Thus, $\widehat{T}$ is a Rota-Baxter operator of weight $\lambda$ on $(A \ltimes_{\rho} V, \bullet)$ if and only if $T$ is a relative Rota-Baxter operator of weight $\lambda$ on $(A, \circ)$ associated to $((V, \circ_{V}); \rho)$.
	The proof is complete.
\end{proof}

%%%%%%%%%%%%%%%%%%%%%%%%%%%%%%%%%%%%%%%%%%%%%%%%%%%%%%%%%%%%%%%%%%%%%%%%%%%%%%%%
\subsection{Extended relative Rota-Baxter operators}

\begin{definition}
	Let $(A, \circ)$ be a Jordan algebra, $(V; \rho)$ be a representation of $(A, \circ)$, and $T, S: V \to A$ be linear maps.
	Then $T$ is called an \textbf{extended relative Rota-Baxter operator with extension $S$} on $(A, \circ)$ associated to $(V; \rho)$ if
	\begin{equation}
		T(u) \circ T(v) - T\big(\rho(T(u))v + \rho(T(v))u \big)  = - S(u) \circ S(v), \;\; \forall u, v \in V. \label{eq:exto}
	\end{equation}
\end{definition}

\begin{remark}
	Let $(A, \circ)$ be a Jordan algebra and regard $(V; \rho)$ as an $A$-module Jordan algebra equipped with the trivial Jordan algebra structure. 
	When $S = 0$, an extended relative Rota-Baxter operator with extension $S$ on $(A, \circ)$ associated to $(V; \rho)$ reduces to a relative Rota-Baxter operator of weight $0$ as defined in Definition~\ref{def:oo}, which is precisely the $\mathcal{O}$-operator introduced in~\cite[Definition~3.5]{hou2013pre}.
\end{remark}

\begin{lemma}\label{lem:rb2eo}
	Let $(A, \circ)$ be a Jordan algebra and $T: A \to A$ be a linear map.
	Then $T$ is a Rota-Baxter operator of weight $\lambda$ on $(A, \circ)$ if and only if $T + \frac{\lambda}{2}\id$ is an extended relative Rota-Baxter operator with extension $\frac{\lambda}{2}\id$ on $(A, \circ)$ associated to $(A; L_{\circ})$.
\end{lemma}
\begin{proof}
	For all $x, y \in A$, we have
	\begin{align*}
		 & (T + \frac{\lambda}{2}\id)(x) \circ (T + \frac{\lambda}{2}\id)(y) - (T + \frac{\lambda}{2}\id)\left( (T + \frac{\lambda}{2}\id)(x) \circ y + x \circ (T + \frac{\lambda}{2}\id)(y)\right) + \frac{\lambda^2}{4} x \circ y \\
		 & = T(x) \circ T(y) - T(T(x) \circ y + x \circ T(y) + \lambda x \circ y).
	\end{align*}
	Therefore, the proof is complete.
\end{proof}

\begin{definition}
	Let $(A, \circ)$ be a Jordan algebra and $(V; \rho)$ be a representation of $(A, \circ)$.
	A linear map $S: V \to A$ is called \textbf{balanced} associated to $(V; \rho)$ if
	\begin{equation}
		\rho(S(u))v = \rho(S(v))u, \;\; \forall u, v \in V. \label{eq:bal}
	\end{equation}
	If, in addition, $S: V \to A$ is a homomorphism of representations from $(V; \rho)$ to $(A; L_{\circ})$, i.e.,
	\begin{equation}
		S(\rho(x)u) = x \circ S(u), \;\; \forall x \in A, \; u \in V, \label{eq:ainv}
	\end{equation}
	then $S$ is called a \textbf{balanced homomorphism of representations} from $(V; \rho)$ to $(A; L_{\circ})$.
\end{definition}

\begin{proposition}\label{prop:bae2aj}
	Let $(A, \circ)$ be a Jordan algebra and $(V; \rho)$ be a representation of $(A, \circ)$.
	Suppose that $S: V \to A$ is a balanced homomorphism of representations from $(V; \rho)$ to $(A; L_{\circ})$.
	Define a binary operation $\cdot_{S}: V \otimes V \to V$ on $V$ by
	\begin{equation}
		u \cdot_{S} v = - 2 \rho(S(u))v, \;\; \forall u, v \in V. \label{eq:bae2aj}
	\end{equation}
	Then $((V, \cdot_{S}); \rho)$ is an $A$-module Jordan algebra.
\end{proposition}
\begin{proof}
	Let $u, v, w, t$.
	Then
	\begin{align*}
		u \cdot_{S} v & = -2\rho(S(u))v \overset{\eqref{eq:bal}}{=} -2\rho(S(v))u = v \cdot_{S} u.
	\end{align*}
	That is, $(V, \cdot_{S})$ is commutative.
	Moreover,
	\begin{align*}
		& ((u \cdot_{S} v) \cdot_{S} t) \cdot_{S} w + ((v \cdot_{S} w) \cdot_{S} t) \cdot_{S} u+((w \cdot_{S} u) \cdot_{S} t) \cdot_{S} v \\
		\overset{\hphantom{\eqref{eq:bal},\eqref{eq:ainv}}}{=}\; & -8\rho(S(w))\rho(S(t)) \rho(S(v))u - 8\rho(S(u))\rho(S(t)) \rho(S(w))v -8\rho(S(v))\rho(S(t)) \rho(S(u))w \\
		\overset{\eqref{eq:bal},\eqref{eq:ainv}}{=}\;
		 & -8\rho(S(w))\rho(S(v) \circ S(u)) t -8\rho(S(u))\rho(S(w) \circ S(v)) t -8\rho(S(v))\rho(S(u) \circ S(w)) t                        \\
		\overset{\;\;\;\;\eqref{repJor2}\;\;\;\;}{=}\; & -8\rho(S(v) \circ S(u)) \rho(S(w)) t -8\rho(S(w) \circ S(v)) \rho(S(u))t - 8\rho(S(u) \circ S(w))\rho(S(v)) t                      \\
		\overset{\;\;\;\eqref{eq:ainv}\;\;\;}{=}\; & -8\rho(S( \rho(S(v)) u )) \rho(S(w)) t -8\rho(S( \rho(S(w)) v )) \rho(S(u)) t - 8\rho(S( \rho(S(u)) w )) \rho(S(v)) t \\
		\overset{\hphantom{\eqref{eq:bal},\eqref{eq:ainv}}}{=}\;  & (u \cdot_{S} v) \cdot_{S}(t \cdot_{S} w)+(v \cdot_{S} w) \cdot_{S}(t \cdot_{S} u)+(w \cdot_{S} u) \cdot_{S}(t \cdot_{S} v).
	\end{align*}
	Hence, $(V, \cdot_{S})$ is a Jordan algebra.
	Furthermore, for all $x \in A$, we have
	\begin{align*}
		& u \cdot_{S} \rho(x)(v \cdot_{S} w) + v \cdot_{S} \rho(x)(w \cdot_{S} u) + w \cdot_{S} \rho(x)(u \cdot_{S} v)\nonumber \\
		\overset{\hphantom{\eqref{eq:bal},\eqref{eq:ainv}}}{=}\; & 4\rho(S( u ))\rho(x) \rho(S(w)) v
		 + 4\rho(S( v ))\rho(x) \rho(S(u)) w
		 + 4\rho(S( w ))\rho(x) \rho(S(v)) u \\
		\overset{\eqref{eq:bal},\eqref{eq:ainv}}{=}\; & 4\rho(x \circ (S(w) \circ S(v)))u
		 + 4\rho(S( v ))\rho(x) \rho(S(w)) u
		 + 4\rho(S( w ))\rho(x) \rho(S(v)) u \\
		\overset{\;\;\;\;\eqref{repJor1}\;\;\;\;}{=}\; & 4\rho(x) \rho (S(v) \circ S(w))u
		 + 4\rho(S( v ))\rho ( S(w)\circ x) u
		 + 4\rho(S( w ))\rho ( x \circ S(v) ) u \\
		\overset{\;\;\;\;\eqref{repJor2}\;\;\;\;}{=}\; & 4\rho (S(v) \circ S(w))\rho(x)u
		 + 4\rho ( S(w)\circ x) \rho(S( v ))u
		 + 4\rho ( x \circ S(v) ) \rho(S( w ))u \\
		\overset{\;\;\;\eqref{eq:ainv}\;\;\;}{=}\; & (\rho(x)u) \cdot_{S} (v \cdot_{S} w) + (\rho(x)w) \cdot_{S} (u \cdot_{S} v) + (\rho(x)v) \cdot_{S} (w \cdot_{S} u).
	\end{align*}
	Therefore, Eq.~\eqref{Ajord1} holds. 
	Similarly, Eqs.~\eqref{Ajord2}-\eqref{Ajord4} hold. 
	Thus, $((V, \cdot_{S}); \rho)$ is an $A$-module Jordan algebra.
\end{proof}

\begin{theorem}\label{thm:eo2ps}
	Let $(A, \circ)$ be a Jordan algebra, $(V; \rho)$ be a representation of $(A, \circ)$ and $T, S: V \to A$ be two linear maps.
	If $S$ is balanced homomorphism of representations from $(V; \rho)$ to $(A; L_{\circ})$ and $T$ is an extended relative Rota-Baxter operator with extension $S$ associated to $(V; \rho)$, then $(V, \cdot_{T})$ is a Jordan algebra, where $\cdot_{T}$ is defined by
	\begin{equation}
		u \cdot_{T} v:= \rho(T(u))v + \rho(T(v))u, \;\; \forall u, v \in V. \label{eq:o2nj}
	\end{equation}
\end{theorem}
\begin{proof}
	For all $u, w \in V$, we have
	\begin{equation*}
		T(u \cdot_{T} v) = T(\rho(T(u))v + \rho(T(v))u) \overset{\eqref{eq:exto}}{=} T(u) \circ T(v) + S(u) \circ S(v),
	\end{equation*}
	and
	\begin{align*}
		&((u \cdot_{T} u) \cdot_{T} w) \cdot_{T} u - (u \cdot_{T} u) \cdot_{T} (w \cdot_{T} u) \\
		\overset{\hphantom{\;\;\eqref{repJor1},\eqref{repJor2}\;\;}}{=}\; & \rho(T((u \cdot_{T} u) \cdot_{T} w))u + \rho(T(u))((u \cdot_{T} u) \cdot_{T} w) - \rho(T(u \cdot_{T} u))(w \cdot_{T} u) - \rho(T(w \cdot_{T} u)) (u \cdot_{T} u) \\
		\overset{\hphantom{\;\;\eqref{repJor1},\eqref{repJor2}\;\;}}{=}\; & \rho(T( \rho(T(u \cdot_{T} u))w + \rho(T(w))(u \cdot_{T} u)  )u + \rho(T(u))( \rho(T(u \cdot_{T} u))w + \rho(T(w))(u \cdot_{T} u) ) \\
		& - 2\rho(T( \rho(T(u))u ))( \rho(T(u))w + \rho(T(w))u ) - 2\rho(T( \rho(T(u))w + \rho(T(w))u )) \rho(T(u))u \\
		\overset{\;\;\;\eqref{eq:exto}\;\;\;}{=}\; & \rho(T(u \cdot_{T} u) \circ T(w))u +  \rho(S(u \cdot_{T} u) \circ S(w))u \\
		& + \rho(T(u)) \rho(T(u) \circ T(u) + S(u) \circ S(u) )w + 2\rho(T(u)) \rho(T(w))\rho(T(u))u \\
		& - \rho(T(u) \circ T(u) + S(u) \circ S(u))( \rho(T(u))w + \rho(T(w))u ) \\
		& - 2\rho(T(u) \circ T(w) + S(u) \circ S(w)) \rho(T(u))u \\
		\overset{\hphantom{\;\;\eqref{repJor1},\eqref{repJor2}\;\;}}{=}\; & \rho((T(u) \circ T(u)) \circ T(w)) u + \rho((S(u) \circ S(u)) \circ T(w)) u + 2 \rho((T(u) \circ S(u)) \circ S(w))u \\
		& + \rho(T(u)) \rho(T(u) \circ T(u))w + \rho(T(u)) \rho(S(u) \circ S(u) )w + 2 \rho(T(u)) \rho(T(w)) \rho(T(u))u \\
		& - \rho(T(u) \circ T(u)) \rho(T(u))w - \rho(T(u) \circ T(u)) \rho(T(w))u - \rho(S(u) \circ S(u)) \rho(T(u))w \\
		& - \rho(S(u) \circ S(u)) \rho(T(w))u - 2 \rho(T(u) \circ T(w)) \rho(T(u)) u - 2 \rho(S(u) \circ S(w)) \rho(T(u)) u \\
		\overset{\;\;\eqref{repJor1},\eqref{repJor2}\;\;}{=}\; & \rho((S(u) \circ S(u)) \circ T(w)) u + 2 \rho((T(u) \circ S(u)) \circ S(w))u + \rho(T(u)) \rho(S(u) \circ S(u) )w \\
		& - \rho(S(u) \circ S(u)) \rho(T(u))w - \rho(S(u) \circ S(u)) \rho(T(w))u - 2 \rho(S(u) \circ S(w)) \rho(T(u)) u \\
		\overset{\eqref{eq:bal},\eqref{eq:ainv}}{=}\; & \rho((S(u) \circ S(u)) \circ T(w)) u + 2 \rho(S(u))\rho(T(u) \circ S(u)) w + \rho(T(u)) \rho(S(u) \circ S(u) )w \\
		& - \rho(S(u) \circ S(u)) \rho(T(u))w - \rho(S(u) \circ S(u)) \rho(T(w))u - 2 \rho(T(u) \circ S(u)) \rho(S(u))w \\
		\overset{\;\;\eqref{repJor1},\eqref{repJor2}\;\;}{=}\; & \rho((S(u) \circ S(u)) \circ T(w)) u - \rho(S(u) \circ S(u)) \rho(T(w))u.
	\end{align*}
	But, for all $u, w$, we have
	\begin{align*}
		& \rho((S(u) \circ S(u)) \circ T(w)) u - \rho(S(u) \circ S(u)) \rho(T(w))u \\
		\overset{\;\;\eqref{repJor1}, \eqref{repJor2}\;\;}{=}\; & 2 \rho(S(u) \circ T(w)) \rho(S(u)) u - 2 \rho(S(u)) \rho(T(w))\rho(S(u)) u \\
		\overset{\eqref{eq:bal},\eqref{eq:ainv}}{=}\; & 2 \rho(S(u) \circ S(u)) \rho(T(w))u - 2 \rho((S(u) \circ S(u)) \circ T(w))u,
	\end{align*}
	that is 
	\begin{equation*}
		3(\rho((S(u) \circ S(u)) \circ T(w)) u - \rho(S(u) \circ S(u)) \rho(T(w))u) = 0.
	\end{equation*}
	Therefore, $((u \cdot_{T} u) \cdot_{T} w) \cdot_{T} u - (u \cdot_{T} u) \cdot_{T} (w \cdot_{T} u) = 0$ and thus $(V, \cdot_{T})$ is a Jordan algebra.
\end{proof}

\begin{corollary}
	Let $(A, \circ)$ be a Jordan algebra and $T: V \to A$ be a Rota-Baxter operator of weight $\lambda$ on $(A, \circ)$.
	Then $(A, \cdot_{T})$ is a Jordan algebra, where $\cdot_{T}$ is defined by
	\begin{equation*}
		x \cdot_{T} y:= T(x) \circ y + x \circ T(y) + \lambda x \circ y, \;\; \forall x, y \in A. 
	\end{equation*}
\end{corollary}
\begin{proof}
	By Lemma~\ref{lem:rb2eo}, $T + \frac{\lambda}{2}\id$ is an extended relative Rota-Baxter operator with extension $\frac{\lambda}{2}\id$ associated to $(A; L_{\circ})$.
	Theorem~\ref{thm:eo2ps} then implies that the binary operation
	\begin{equation*}
		x \cdot y = (T + \frac{\lambda}{2}\id)(x) \circ y + x \circ (T + \frac{\lambda}{2}\id)(y) = x \cdot_{T} y
	\end{equation*}
	defines a Jordan algebra structure on $A$.
	The proof is complete.
\end{proof}

Let $A, V$ be vector spaces, and $\pi_{\pm}: V \to A$ be two linear maps.
Set
\begin{equation*}
	T:= \frac{1}{2}(\pi_{+} + \pi_{-}), \;\;
	S := \frac{1}{2}(\pi_{+} - \pi_{-}),
\end{equation*}
which are called the \textbf{symmetrizer} and \textbf{antisymmetrizer} of $\pi_{\pm}$, respectively.
Note that $\pi_{\pm}$ can be resolved from $T$ and $S$ by $\pi_{\pm} = T \pm S$.

\begin{lemma}\label{lem:iphr}
	Let $(A, \circ)$ be a Jordan algebra and $(V; \rho)$ be a representation of $(A, \circ)$. 
	Let $\pi_{\pm}: V \to A$ be two linear maps, and $T, S$ be the symmetrizer and antisymmetrizer of $\pi_{\pm}$ respectively.
	Suppose that $S$ is a homomorphism of representations from $(V; \rho)$ to $(A; L_{\circ})$.
	Then the following conditions are equivalent.
	\begin{enumerate}[label=(\roman*)]
		\item\label{it:eo2h1}
		      $T$ is an extended relative Rota-Baxter operator with extension $S$ associated $(V; \rho)$.

		\item\label{it:eo2h2}
		      $\pi_{+}(u) \circ \pi_{+}(v)  = \pi_{+}(\rho(T(u))v + \rho(T(v))u)$ holds for all $u, v \in V$.

		\item\label{it:eo2h3}
		      $\pi_{-}(u) \circ \pi_{-}(v) = \pi_{-}(\rho(T(u))v + \rho(T(v))u)$ holds for all $u, v \in V$.

	\end{enumerate}
\end{lemma}
\begin{proof}
	For all $u, v \in V$, we have
	\begin{align*}
		& \pi_{+}(u) \circ \pi_{+}(v) - \pi_{+}(\rho(T(u))v + \rho(T(v))u)                                          \\
		\overset{\hphantom{\eqref{eq:ainv}}}{=}\; &  (T+S)(u) \circ (T+S)(v) - (T+S)(\rho(T(u))v + \rho(T(v))u )       \\
		\overset{\hphantom{\eqref{eq:ainv}}}{=}\; &  T(u) \circ T(v) + S(u) \circ S(v) - T(\rho(T(u))v + \rho(T(v))u)  \\
		& + T(u) \circ S(v) + S(u) \circ T(v) - S(\rho(T(u))v + \rho(T(v))u)                                  \\
		\overset{\eqref{eq:ainv}}{=}\; & T(u) \circ T(v) + S(u) \circ S(v) - T(\rho(T(u))v + \rho(T(v))u).
	\end{align*}
	Hence, we show the equivalence between \ref{it:eo2h1} and \ref{it:eo2h2}.
	The equivalence \ref{it:eo2h1} $\Longleftrightarrow$ \ref{it:eo2h3} follows similarly.
\end{proof}

The following result provides equivalent characterizations of extended relative Rota-Baxter operators via the symmetrizer-antisymmetrizer decomposition.

\begin{theorem}\label{thm:opm}
	Let $(A, \circ)$ be a Jordan algebra and $(V; \rho)$ be a representation of $(A, \circ)$. 
	Let $\pi_{\pm}: V \to A$ be two linear maps, and $T, S$ be the symmetrizer and antisymmetrizer of $\pi_{\pm}$ respectively.
	Suppose that $S$ is a balanced homomorphism of representations from $(V; \rho)$ to $(A; L_{\circ})$.
	Then, the following conditions are equivalent.
	\begin{enumerate}[label=(\roman*)]
		\item\label{it:oeoheo}
		      $T$ is an extended relative Rota-Baxter operator with extension $S$ on $(A, \circ)$ associated to $(V; \rho)$.

		\item\label{it:oeohh}
		      $(V, \cdot_{T})$ is a Jordan algebra, and $\pi_{+}$ (resp. $\pi_{-}$) is a homomorphism of Jordan algebras between $(V, \cdot_{T})$ and $(A, \circ)$, where $\cdot_{T}: V \otimes V \to V$ is defined by Eq.~\eqref{eq:o2nj}.

		\item\label{it:oeoho}
		      $\pi_{+}$ (resp. $\pi_{-}$) is a relative Rota-Baxter operator of weight $1$ (resp. $-1$) on $(A, \circ)$ associated to $((V, \cdot_{S}); \rho)$, where $((V, \cdot_{S}); \rho)$ is the $A$-module Jordan algebra obtained in Proposition~\ref{prop:bae2aj}.

	\end{enumerate}
\end{theorem}
\begin{proof}
	\ref{it:oeoheo} $\Longrightarrow$ \ref{it:oeohh}.
	By Theorem~\ref{thm:eo2ps}, $(V, \cdot_{T})$ is a Jordan algebra.
	Then, Lemma~\ref{lem:iphr} shows that $\pi_{+}$ is a homomorphism of Jordan algebras.

	\ref{it:oeohh} $\Longrightarrow$ \ref{it:oeoheo}.
	It follows from Lemma~\ref{lem:iphr}.

	\ref{it:oeoheo} $\Longleftrightarrow$ \ref{it:oeoho}.
	As in the proof of Lemma~\ref{lem:iphr}, we have
	\begin{equation*}
		T(u) \circ T(v) + S(u) \circ S(v) - T(\rho(T(u))v + \rho(T(v))u) = \pi_{+}(u) \circ \pi_{+}(v) - \pi_{+}\big( \rho(T(u))v + \rho(T(v))u \big).
	\end{equation*}
	On the other hand, we have
	\begin{align*}
		 & \rho(\pi_{+}(u))v + \rho(\pi_{+}(v))u + u \cdot_{S} v =\rho(\pi_{+}(u))v + \rho(\pi_{+}(v))u - 2 \rho(S(u))v \\
		 \overset{\hphantom{\eqref{eq:bal}}}{=}\; & \rho(T(u))v + \rho(T(v))u + \rho(S(u))v + \rho(S(v))u - 2 \rho(S(u))v  \\
		 \overset{\eqref{eq:bal}}{=}\; & \rho(T(u))v + \rho(T(v))u.
	\end{align*}
	Therefore,
	\begin{align*}
		&T(u) \circ T(v) + S(u) \circ S(v) - T(\rho(T(u))v + \rho(T(v))u) \\
		=\; & \pi_{+}(u) \circ \pi_{+}(v) - \pi_{+}\big( \rho(\pi_{+}(u))v + \rho(\pi_{+}(v))u + u \cdot_{S} v \big).
	\end{align*}
	This proves the equivalence between  \ref{it:oeoheo}  and  \ref{it:oeoho}  for $\pi_{+}$.
	The case of $\pi_{-}$ follows similarly.
\end{proof}

%%%%%%%%%%%%%%%%%%%%%%%%%%%%%%%%%%%%%%%%%%%%%%%%%%%%%%%%%%%%%%%%%%%%%%%%%%%%%%%%
%%%%%%%%%%%%%%%%%%%%%%%%%%%%%%%%%%%%%%%%%%%%%%%%%%%%%%%%%%%%%%%%%%%%%%%%%%%%%%%%
%%%%%%%%%%%%%%%%%%%%%%%%%%%%%%%%%%%%%%%%%%%%%%%%%%%%%%%%%%%%%%%%%%%%%%%%%%%%%%%%
\section{Quasi-triangular Jordan D-bialgebras and extended relative Rota-Baxter operators}\label{sec:JYBE}

In this section, we recall the notion of the Jordan Yang-Baxter equations, which arise naturally in the study of Jordan D-bialgebras.
We then introduce quasi-triangular Jordan D-bialgebras, constructing from solutions of the Jordan Yang-Baxter equations with invariant symmetric part.
Finally, we establish the relationship between extended relative Rota-Baxter operators and such solutions.

%%%%%%%%%%%%%%%%%%%%%%%%%%%%%%%%%%%%%%%%%%%%%%%%%%%%%%%%%%%%%%%%%%%%%%%%%%%%%%%%
\subsection{Quasi-triangular Jordan D-bialgebras}
Let $\mathbb{S}_{4}$ be the symmetric group of order $4$ acting naturally on $V^{\otimes 4}$ by permuting tensor factors
\begin{equation*}
	\sigma (v_{1} \otimes v_{2} \otimes v_{3} \otimes v_{4}) = v_{\sigma(1)} \otimes v_{\sigma(2)} \otimes v_{\sigma(3)} \otimes v_{\sigma(4)}, \quad \forall v_{i} \in V,
\end{equation*}
and let $\Omega := \{\id^{\otimes 4}, \sigma_{124}, \sigma_{142}\}$ be the subgroup of $\mathbb{S}_{4}$ generated by the indicated permutations.

\begin{definition}
	A \textbf{Jordan coalgebra} is a pair $(A, \Delta)$, where $A$ is a vector space and $\Delta: A \to A \otimes A$ is a linear map satisfying the following equalities for all $x \in A$:
	\begin{align*}
		\tau \Delta(x) - \Delta(x) & = 0, \\
		\sum_{\sigma\in \Omega}\sigma \Big((\Delta \otimes \id \otimes \id)(\Delta \otimes \id - \id \otimes \Delta)\Delta (x)\Big) & = 0,
	\end{align*}
	where $\tau: V \otimes V \to V\otimes V$ is the exchange operator defined by $\tau(u \otimes v) = v \otimes u$ for all $u, v \in V$.
\end{definition}

\begin{remark}
	Let $A$ be a finite-dimensional vector space and $\Delta: A \to A \otimes A$ be a linear map. 
	Let $\circ_{A^*}: A^* \otimes A^* \to A^*$ be the linear dual of $\Delta$ defined by
	\begin{align*}
		\langle a^* \circ_{A^{*}} b^*, x \rangle := \langle a^* \otimes b^*, \Delta(x) \rangle,\;\forall x \in A, \; a^*,b^* \in A^{*}.
	\end{align*}
	Then $(A, \Delta)$ is a Jordan coalgebra if and only if $(A^*, \circ_{A^*})$ is a Jordan algebra~\cite[Lemma 4.4]{hou2022new}.
\end{remark}

\begin{definition}
	A {\bf Jordan D-bialgebra} is a triple $(A, \circ, \Delta)$, where $(A, \circ)$ is a Jordan algebra and $(A,\Delta)$ is a Jordan coalgebra satisfying the following compatible conditions for all $x, y, z \in A$:
	\begin{align}
		& \Delta((x \circ y) \circ z) + (y \otimes 1) \circ ((z \otimes 1) \circ \Delta(x))
		+ (1 \otimes y) \circ ((1 \otimes z) \circ \Delta(x)) \notag                                                                                                                                         \\
		    & + (x \otimes 1) \circ ((z \otimes 1) \circ \Delta(y)) + (1 \otimes x) \circ ((1 \otimes z) \circ \Delta(y)) + (x \otimes y + y \otimes x) \circ \Delta(z) \notag                               \\
		=\; & (x \otimes 1) \circ \Delta(y \circ z) + (y \otimes 1) \circ \Delta(z \circ x) + (z \otimes 1) \circ \Delta(x \circ y) \notag                                                                   \\
		    & + (1 \otimes (y \circ z)) \circ \Delta(x) + (1 \otimes (z \circ x)) \circ \Delta(y) + (1 \otimes (x \circ y)) \circ \Delta(z), \label{eq:jdb1}                                                 \\
		    & (\Delta \otimes \id)((x \otimes 1) \circ \Delta(y)) + (\id \otimes \Delta)((1 \otimes x) \circ \Delta(y)) + (\tau \otimes \id)((\id \otimes \Delta)((1 \otimes x) \circ \Delta(y))) \notag   \\
		    & + (\Delta \otimes \id)((1 \otimes y) \circ \Delta(x))
		+ (\id \otimes \Delta)((y \otimes 1) \circ \Delta(x)) + (\tau \otimes \id)((\text{id} \otimes \Delta)((y \otimes 1) \circ \Delta(x))) \notag                                                       \\
		=\; & (\Delta \otimes \text{id})\Delta(x \circ y) + (1 \otimes x \otimes 1) \circ ((\text{id} \otimes \Delta)\Delta(y)) + (1 \otimes y \otimes 1) \circ ((\text{id} \otimes \Delta)\Delta(x)) \notag \\
		    & + (x \otimes 1 \otimes 1) \circ ((\text{id} \otimes \tau)((\Delta \otimes \text{id})\Delta(y)))
		+ (y \otimes 1 \otimes 1) \circ ((\text{id} \otimes \tau)((\Delta \otimes \text{id})\Delta(x))) \notag                                                                                             \\
		    & + (1 \otimes \Delta(x)) \circ ((\text{id} \otimes \tau)((\Delta(y) \otimes 1))) + (1 \otimes \Delta(y)) \circ ((\text{id} \otimes \tau)((\Delta(x) \otimes 1))). \label{eq:jdb2}
	\end{align}
\end{definition}

\begin{definition}[\cite{zhelyabin2000class}]
	A Jordan D-bialgebra $(A, \circ, \Delta)$ is called \textbf{coboundary} if there exists $r \in A \otimes A$ such that:
	\begin{equation}
		\Delta(x) = \Delta_r(x) := ( \id \otimes L_{\circ}(x) - L_{\circ}(x) \otimes \id )(r), \;\; \forall x \in A. \label{eq:jcbd}
	\end{equation}
\end{definition}

\begin{definition}[\cite{zhelyabin2000class}]
	Let $(A, \circ)$ be a Jordan algebra. 
	Then $r \in A \otimes A$ is called a solution of the {\bf Jordan Yang-Baxter equation (JYBE)} in $(A, \circ)$ if
	\begin{equation*}
		\mathbf{J}(r) := r_{12} \circ r_{13} - r_{12} \circ r_{23} + r_{13} \circ r_{23} = 0.
	\end{equation*}
\end{definition}

\begin{proposition}[\cite{hou2022new}]\label{prop:cobdppb}
	Let $(A, \circ)$ be a Jordan algebra and $r = \sum\nolimits_{i} x_i \otimes y_i \in A \otimes A$.
	Define a linear map $\Delta: A \to A \otimes A$ by Eq.~\eqref{eq:jcbd}.
	Then
	\begin{enumerate}[label=(\roman*)]
		\item\label{tdb:co1}
		      $(A,\Delta)$ is a Jordan coalgebra if and only if the following equalities hold for all $x \in A$:
		      \begin{align}
			      ( \id \otimes L_{\circ}(x)-L_{\circ}(x) \otimes \id )(r + \tau(r)) & = 0,  \label{CJBco:1} \\
			      \sum_{\sigma\in \Omega}\sigma(A(1) + A(3))                         & = 0, \label{CJBco:2}
		      \end{align}
		      where
		      \begin{align*}
			      A(1) =\; & (\Delta \otimes \id \otimes \id)(L_{\circ}(x) \otimes \id \otimes \id - \id \otimes \id \otimes L_{\circ}(x)) \mathbf{J}(r)                                                          \\
			               & + \sum_{j} \bigl\{ -\big((\id \otimes \id \otimes L_{\circ}(x \circ x_j)) (r_{12} \circ r_{13} - r_{12} \circ r_{23} + r_{31} \circ r_{32}) \big) \otimes y_j                        \\
			               & + \big( (\id \otimes \id \otimes L_{\circ}(x_j) L_{\circ}(x) ) (r_{12} \circ r_{13} - r_{12} \circ r_{23} + r_{31} \circ r_{32}) \big) \otimes y_j                                   \\
			               & + \sigma_{34} \bigl[ \big( (\id \otimes \id \otimes L_{\circ}(x) L_{\circ}(y_j)) (r_{12} \circ r_{13} - r_{12} \circ r_{23} + r_{31} \circ r_{32}) \big) \otimes x_j \bigr]          \\
			               & - \sigma_{34} \bigl[ \big( (\id \otimes \id \otimes L_{\circ}(y_j) L_{\circ}(x)) (r_{12} \circ r_{13} - r_{12} \circ r_{23} + r_{31} \circ r_{32}) \big) \otimes x_j \bigr] \bigr\}, \\
			      A(3) =\; & \sum_{j} - y_j \otimes \bigl\{ (\id \otimes \id \otimes L_{\circ}(x)) \big([(\id \otimes \id \otimes L_{\circ}(x_j))r_{31}] \circ (r_{23} + r_{32}) \big) \bigr\}                    \\
			               & + \sum_j y_j \otimes \bigl\{ [(\id \otimes \id \otimes L_{\circ}(x) L_{\circ}(x_j))r_{31}] \circ (r_{23} + r_{32}) \bigr\}                                                           \\
			               & + \sum_j \bigl\{ (r_{12} + r_{21}) \circ [(\id \otimes (L_{\circ}(x) L_{\circ}(x_j) - L_{\circ}(x \circ x_j)) \otimes \id) r_{23} ]\bigr\} \otimes y_j.
		      \end{align*}
		\item\label{tdb:co2} $(A, \circ, \Delta)$ is a coboundary Jordan D-bialgebra if and only if Eqs.~\eqref{CJBco:1}-\eqref{CJBco:2} and the following equality holds for all $x, y \in A$:
		      \begin{align}
			       & \bigl( L_{\circ}(y) L_{\circ}(x) \otimes \id \otimes \id + \id \otimes \id \otimes L_{\circ}(y) L_{\circ}(x) - \id \otimes \id \otimes L_{\circ}(x \circ y) \bigr)\mathbf{J}(r) \notag   \\
			       & - \bigl( L_{\circ}(x) \otimes L_{\circ}(y) \otimes \id + L_{\circ}(y) \otimes L_{\circ}(x) \otimes \id \bigr) \mathbf{J}(r)
			      + (\id \otimes L_{\circ}(y) L_{\circ}(x) \otimes \id)\mathbf{J}(r) \notag                                                                                           \\
			       & + \bigl( L_{\circ}(x) \otimes L_{\circ}(y) \otimes \id + L_{\circ}(y) \otimes L_{\circ}(x) \otimes \id \bigr) ( r_{13} \circ r_{23}+ r_{13} \circ r_{32})\notag                  \\
			       & - (\id \otimes L_{\circ}(y) L_{\circ}(x) \otimes \id)( r_{12} \circ r_{13} + r_{21} \circ r_{13} ) \notag                                                        \\
			       & + r_{13} \circ \bigl( (L_{\circ}(x) L_{\circ}(y) \otimes \id \otimes \id )(r_{12} + r_{21}) \bigr) \notag                                                        \\
			       & - (L_{\circ}(x) \otimes \id \otimes \id) \bigl( r_{13} \circ [(\id \otimes \id \otimes L_{\circ}(y))(r_{23} + r_{32})] \bigr) \notag                             \\
			       & - \bigl( L_{\circ}(y) \otimes \id \otimes \id \bigr) \bigl( r_{13} \circ [(\id \otimes \id \otimes L_{\circ}(x))(r_{23} + r_{32})] \bigr) = 0. \label{eq:hou3.6}
		      \end{align}
	\end{enumerate}
\end{proposition}

\begin{lemma}\label{lem:pr}
	Let $(A, \circ)$ be a Jordan algebra and $r = \sum\nolimits_{i} x_i \otimes y_i \in A \otimes A$.
	Then the following equality holds for all $x \in A$:
	\begin{align}
		\sum_{\sigma\in \Omega} \sum_{i, j, k} \sigma\Big(x_i \otimes x_j \otimes \big((x \circ (y_i \circ y_j) )\circ y_k - (x \circ y_k) \circ (y_i \circ y_j) \big)\otimes x_k\Big) = 0. \label{eq:pr1}
	\end{align}
\end{lemma}
\begin{proof}
	For all $x \in A$, we have
	\begin{align*}
		& \sum_{\sigma\in \Omega} \sum_{i, j, k} \sigma\Big(x_i \otimes x_j \otimes \big((x \circ (y_i \circ y_j) )\circ y_k - (x \circ y_k) \circ (y_i \circ y_j) \big)\otimes x_k\Big) \\
		\overset{\hphantom{\eqref{eq:jord}}}{=} \; & \sum_{i, j, k} \Big(x_i \otimes x_j \otimes \big((x \circ (y_i \circ y_j) )\circ y_k - (x \circ y_k) \circ (y_i \circ y_j) \big)\otimes x_k \\
		& + x_j \otimes x_k \otimes \big((x \circ (y_i \circ y_j) )\circ y_k - (x \circ y_k) \circ (y_i \circ y_j) \big)\otimes x_i \\
		& + x_k \otimes x_i \otimes \big((x \circ (y_i \circ y_j) )\circ y_k - (x \circ y_k) \circ (y_i \circ y_j) \big)\otimes x_j\Big) \\
		\overset{\hphantom{\eqref{eq:jord}}}{=} \; & \sum_{i, j, k} \Big(x_i \otimes x_j \otimes \big((x \circ (y_i \circ y_j) )\circ y_k - (x \circ y_k) \circ (y_i \circ y_j) \big)\otimes x_k \\
		& + x_i \otimes x_j \otimes \big((x \circ (y_k \circ y_i) )\circ y_j - (x \circ y_j) \circ (y_k \circ y_i) \big)\otimes x_k \\
		& + x_i \otimes x_j \otimes \big((x \circ (y_j \circ y_k) )\circ y_i - (x \circ y_i) \circ (y_j \circ y_k) \big)\otimes x_k\Big) \\
		\overset{\eqref{eq:jord}}{=}\; & 0.
	\end{align*}
	The proof is complete.
\end{proof}

\begin{lemma}\label{lem:sfh}
	Let $(A, \circ)$ be a Jordan algebra and $r = \sum\nolimits_{i} x_i \otimes y_i \in A \otimes A$.
	Define a linear map $\Delta: A \to A \otimes A$ by Eq.~\eqref{eq:jcbd}.
	Suppose that Eq.~\eqref{CJBco:1} holds for all $x \in A$.
	Then
	\begin{enumerate}[label=(\roman*)]
		\item\label{it:hr1}
		    Eq.~\eqref{CJBco:2} holds for all $x \in A$ if and only if for all $x \in A$,
		    \begin{align}
			    & (\Delta \otimes \id \otimes \id)(L_{\circ}(x) \otimes \id \otimes \id - \id \otimes \id \otimes L_{\circ}(x)) (\mathbf{J}(r)) \notag                                                                          \\
			    & + \sum_{i} \big( (-\id \otimes \id \otimes L_{\circ}(x \circ x_i) + \id \otimes \id \otimes L_{\circ}(x_i) L_{\circ}(x)) (\mathbf{J}(r)) \big) \otimes y_i \notag                                             \\
			    & + \sum_{i} \sigma_{34} \bigl[ \big( (\id \otimes \id \otimes L_{\circ}(x) L_{\circ}(y_i) - \id \otimes \id \otimes L_{\circ}(y_i) L_{\circ}(x) ) (\mathbf{J}(r)) \big) \otimes x_i \bigr] = 0. \label{eq:jcs}
		    \end{align}

		\item\label{it:hr2}
		    Eq.~\eqref{eq:hou3.6} holds for all $x, y \in A$ if and only if for all $x,y \in A$,
		    \begin{align}
			    & ( L_{\circ}(x) L_{\circ}(y) \otimes \id \otimes \id + \id \otimes \id \otimes L_{\circ}(x) L_{\circ}(y) - \id \otimes \id \otimes L_{\circ}(x \circ y) \bigr)( \mathbf{J}(r)) \notag        \\
			    & + ( \id \otimes L_{\circ}(x)L_{\circ}(y) \otimes \id - L_{\circ}(x) \otimes L_{\circ}(y) \otimes \id - L_{\circ}(y) \otimes L_{\circ}(x) \otimes \id ) ( \mathbf{J}(r)) = 0. \label{eq:jbcs}
		    \end{align}
	\end{enumerate}
\end{lemma}
\begin{proof}
	\ref{it:hr1}. Let $x, y \in A$. On one hand, by Lemma~\ref{lem:pr}, we have
	\begin{eqnarray*}
		&&\sum_{\sigma \in \Omega}\sigma\bigg(\sum_{j} \bigl\{ -\big((\id \otimes \id \otimes L_{\circ}(x \circ x_j)) (-r_{13} \circ r_{23} + r_{31} \circ r_{32}) \big) \otimes y_j \\
		&&\quad + \big( (\id \otimes \id \otimes L_{\circ}(x_j) L_{\circ}(x) ) (-r_{13} \circ r_{23} + r_{31} \circ r_{32}) \big) \otimes y_j \\
		&&\quad + \sigma_{34} \bigl[ \big( (\id \otimes \id \otimes L_{\circ}(x) L_{\circ}(y_j)) (-r_{13} \circ r_{23} + r_{31} \circ r_{32}) \big) \otimes x_j \bigr] \\
		&&\quad - \sigma_{34} \bigl[ \big( (\id \otimes \id \otimes L_{\circ}(y_j) L_{\circ}(x)) (-r_{13} \circ r_{23} + r_{31} \circ r_{32}) \big) \otimes x_j \bigr] \bigr\} + \underbrace{A(3)} \bigg) \\
		&&\overset{\eqref{eq:pr1}}{=} \sum_{\sigma \in \Omega}\sigma\bigg( \sum_{i,j, k} \{ x_i \otimes x_k \otimes (x \circ x_j) \circ (y_i \circ y_k) \otimes y_j - x_i \otimes x_k \otimes x_j \circ (x \circ (y_i \circ y_k)) \otimes y_j \\
		&&\quad\quad - x_i \otimes x_k \otimes x_j \otimes x \circ (y_j \circ (y_i \circ y_k)) + x_i \otimes x_k \otimes x_j \otimes y_j \circ (x \circ (y_i \circ y_k)) \\
		&&\quad\quad - y_k \otimes y_i \otimes y_j \otimes x \circ (x_j \circ (x_i \circ x_k)) + y_k \otimes y_i \otimes y_j \otimes x_j \circ (x \circ (x_i \circ x_k)) \\
		&&\quad\quad \underbrace{ + x_i \otimes y_i \circ (x \circ (x_k \circ x_j)) \otimes y_j \otimes y_k - x_i \otimes y_i \circ (x_j \circ (x \circ x_k)) \otimes y_j \otimes y_k} \\
		&&\quad\quad \underbrace{ + y_i \otimes x_i \circ (x \circ (x_k \circ x_j)) \otimes y_j \otimes y_k - y_i \otimes x_i \circ (x_j \circ (x \circ x_k)) \otimes y_j \otimes y_k \}} \bigg) \\
		&&\overset{\eqref{eq:pr1}}{=}\sum_{\sigma \in \Omega}\sigma\bigg( \sum_{i,j, k} \{ \uline{ x_i \otimes x_k \otimes (x \circ x_j) \circ (y_i \circ y_k) \otimes y_j} \uwave{ - x_i \otimes x_k \otimes x_j \circ (x \circ (y_i \circ y_k)) \otimes y_j} \\
		&&\quad\quad + \uuline{ \uline{x_i \otimes x_k \otimes (x \circ y_j) \circ (y_i \circ y_k) \otimes x_j} \uwave{ - x_i \otimes x_k \otimes (x \circ (y_i \circ y_k) )\circ y_j \otimes x_j} }  \;\text{$\leftarrow$this line equal to 0}\\
		&&\quad\quad \uline{ - x_i \otimes x_k \otimes x_j \otimes x \circ (y_j \circ (y_i \circ y_k)) } + \uwave{ x_i \otimes x_k \otimes x_j \otimes y_j \circ (x \circ (y_i \circ y_k)) } \\
		&&\quad\quad \dashuline{ - y_k \otimes y_i \otimes y_j \otimes x \circ (x_j \circ (x_i \circ x_k)) + y_k \otimes y_i \otimes y_j \otimes x_j \circ (x \circ (x_i \circ x_k)) } \\
		&&\quad\quad + x_i \otimes y_i \circ (x \circ (x_k \circ x_j)) \otimes y_j \otimes y_k - x_i \otimes y_i \circ (x_j \circ (x \circ x_k)) \otimes y_j \otimes y_k \\
		&&\quad\quad + \dashuline{ y_k \otimes y_i \otimes y_j \otimes x_i \circ (x \circ (x_k \circ x_j)) - y_k \otimes y_i \otimes y_j \otimes x_i \circ (x_j \circ (x \circ x_k)) } \} \bigg) \\
		&&\underset{\eqref{CJBco:1}}{\overset{\eqref{ID:jord2}}{=}} \sum_{\sigma \in \Omega}\sigma\bigg( \sum_{i,j, k} \{ \uline{ x_i \otimes x_k \otimes y_j \otimes x \circ (x_j \circ (y_i \circ y_k)) } \uwave{ - x_i \otimes x_k \otimes y_j \otimes x_j \circ (x \circ (y_i \circ y_k)) } \\
		&&\quad\quad + \uuline{ x_i \otimes y_i \circ (x \circ (x_k \circ x_j)) \otimes y_j \otimes y_k - x_i \otimes y_i \circ (x_j \circ (x \circ x_k)) \otimes y_j \otimes y_k } \\
		&&\quad\quad \dashuline{ - y_k \otimes y_i \otimes y_j \otimes x_k \circ (x \circ (x_i \circ x_j)) + y_k \otimes y_i \otimes y_j \otimes x_k \circ (x_j \circ (x \circ x_i)) } \} \bigg) \\
		&&= \sum_{\sigma \in \Omega}\sigma\bigg( \sum_{i,j, k} \{ x_i \otimes x_k \otimes y_j \otimes x \circ (x_j \circ (y_i \circ y_k)) - x_i \otimes x_k \otimes y_j \otimes x_j \circ (x \circ (y_i \circ y_k)) \\
		&&\quad\quad \uuline{ - x_i \otimes x_j \circ (x \circ (x_k \circ y_i)) \otimes y_j \otimes y_k + x_i \otimes x \circ (x_j \circ (y_i \circ x_k)) \otimes y_j \otimes y_k }\\
		&&\quad\quad \uuline{ \uwave{ - x_i \otimes x_k \circ (x \circ (y_i \circ x_j)) \otimes y_j \otimes y_k} + \dashuline{ x_i \otimes x_k \circ (x_j \circ (x \circ y_i)) \otimes y_j \otimes y_k} }\\
		&&\quad\quad \uwave{ - y_k \otimes y_i \otimes y_j \otimes x_k \circ (x \circ (x_i \circ x_j))} + \dashuline{ y_k \otimes y_i \otimes y_j \otimes x_k \circ (x_j \circ (x \circ x_i))} \} \bigg) \\
		&&\overset{\eqref{CJBco:1}}{=}\sum_{\sigma \in \Omega}\sigma\bigg( \sum_{i,j, k} \{ - x_i \otimes x_j \circ (x \circ (x_k \circ y_i)) \otimes y_j \otimes y_k + x_i \otimes x \circ (x_j \circ (y_i \circ x_k)) \otimes y_j \otimes y_k \\
		&&\quad\quad + x_i \otimes x_k \otimes y_j \otimes x \circ (x_j \circ (y_i \circ y_k)) - x_i \otimes x_k \otimes y_j \otimes x_j \circ (x \circ (y_i \circ y_k)) \\
		&&\quad\quad \uwave{ - x_j \circ (x \circ (x_k \circ x_i)) \otimes y_i \otimes y_j \otimes y_k} + \dashuline{ x \circ (x_j \circ (x_k \circ x_i)) \otimes y_i \otimes y_j \otimes y_k} \\
		&&\quad\quad \uwave{ - x_j \circ (x \circ (x_k \circ y_i)) \otimes x_i \otimes y_j \otimes y_k} + \dashuline{ x \circ (x_j \circ (x_k \circ y_i)) \otimes x_i \otimes y_j \otimes y_k} \} \bigg) \\
		&&\overset{\hphantom{\eqref{eq:pr1}}}{=} \sum_{\sigma \in \Sigma} \sum_{i,j, k} \sigma\bigg( x_i \otimes x \circ (x_j \circ (y_i \circ x_k)) \otimes y_j \otimes y_k- x_i\otimes x_j \circ (x \circ (x_k \circ y_i)) \otimes y_j \otimes y_k\\
		&&\quad\quad + x_i \otimes x \circ (x_j \circ (y_i \circ y_k)) \otimes y_j \otimes x_k - x_i \otimes x_j \circ (x \circ (y_i \circ y_k)) \otimes y_j \otimes x_k \\
		&&\quad\quad + x \circ (x_j \circ (x_k \circ x_i)) \otimes y_i \otimes y_j \otimes y_k - x_j \circ (x \circ (x_k \circ x_i)) \otimes y_i \otimes y_j \otimes y_k \\
		&&\quad\quad + x \circ (x_j \circ (x_k \circ y_i)) \otimes x_i \otimes y_j \otimes y_k - x_j \circ (x \circ (x_k \circ y_i)) \otimes x_i \otimes y_j \otimes y_k \bigg) \\
		&&\overset{\hphantom{\eqref{eq:pr1}}}{=} : \mathfrak{S}.
	\end{eqnarray*}
	On the other hand, we have
	\begin{eqnarray*}
		&\mathfrak{S}&\overset{\eqref{ID:jord2}}{=}\sum_{\sigma \in \Omega}\sigma\bigg( \sum_{i,j, k} \{ \uwave{ x_i \otimes y_i \circ (x \circ (x_j \circ x_k)) \otimes y_j \otimes y_k} \dashuline{ - x_i \otimes y_i \circ (x_j \circ (x \circ x_k)) \otimes y_j \otimes y_k} \\
		&&\quad\quad + \uline{ x_i \otimes x_k \circ (x \circ (x_j \circ y_i)) \otimes y_j \otimes y_k} - x_i \otimes x_k \circ (x_j \circ (x \circ y_i)) \otimes y_j \otimes y_k \\
		&&\quad\quad - x_i \otimes y_i \circ (x_j \circ (x \circ y_k)) \otimes y_j \otimes x_k + \uline{ x_i \otimes y_i \circ (x \circ (x_j \circ y_k)) \otimes y_j \otimes x_k} \\
		&&\quad\quad - x_i \otimes y_k \circ (x_j \circ (x \circ y_i)) \otimes y_j \otimes x_k + \uline{ x_i \otimes y_k \circ (x \circ (x_j \circ y_i)) \otimes y_j \otimes x_k} \\
		&&\quad\quad + \uwave{ x_i \circ (x \circ (x_j \circ x_k)) \otimes y_i \otimes y_j \otimes y_k} \dashuline{ - x_i \circ (x_j \circ (x \circ x_k)) \otimes y_i \otimes y_j \otimes y_k} \\
		&&\quad\quad + \uwave{x_k \circ (x \circ (x_j \circ x_i)) \otimes y_i \otimes y_j \otimes y_k} \dashuline{ - x_k \circ (x_j \circ (x \circ x_i)) \otimes y_i \otimes y_j \otimes y_k} \\
		&&\quad\quad + \uline{ x_k \circ (x \circ (x_j \circ y_i)) \otimes x_i \otimes y_j \otimes y_k} - x_k \circ (x_j \circ (x \circ y_i)) \otimes x_i \otimes y_j \otimes y_k \\
		&&\quad\quad + \uwave{y_i \circ (x \circ (x_j \circ x_k)) \otimes x_i \otimes y_j \otimes y_k} \dashuline{ - y_i \circ (x_j \circ (x \circ x_k)) \otimes x_i \otimes y_j \otimes y_k} \} \bigg) \\
		&&\overset{\eqref{CJBco:1}}{=} 2\sum_{\sigma \in \Omega}\sigma\bigg( \sum_{i,j, k} \{ \uwave{ \textcolor{blue}{x_i \circ (x \circ (x_j \circ x_k)) \otimes y_i \otimes y_j \otimes y_k} } \dashuline{ \textcolor{green}{ - x_i \circ (x_j \circ (x \circ x_k)) \otimes y_i \otimes y_j \otimes y_k} } \\
		&&\quad\quad + \uwave{ \textcolor{red}{ y_i \circ (x \circ (x_j \circ x_k)) \otimes x_i \otimes y_j \otimes y_k} } \dashuline{ - y_i \circ (x_j \circ (x \circ x_k)) \otimes x_i \otimes y_j \otimes y_k} \\
		&&\quad\quad + \uline{ \textcolor{blue}{x_i \otimes x_k \circ (x \circ (x_j \circ y_i)) \otimes y_j \otimes y_k} } \textcolor{green}{ - x_i \otimes x_k \circ (x_j \circ (x \circ y_i)) \otimes y_j \otimes y_k} \\
		&&\quad\quad + \uline{ \textcolor{red}{x_i \otimes y_k \circ (x \circ (x_j \circ y_i)) \otimes y_j \otimes x_k} } - x_i \otimes y_k \circ (x_j \circ (x \circ y_i)) \otimes y_j \otimes x_k \} \bigg) \\
		&&\overset{\eqref{CJBco:1}}{=} -2\sum_{\sigma \in \Omega}\sigma\bigg( \sum_{i,j, k} \{ x \circ (x_j \circ (y_k \circ x_i)) \otimes y_i \otimes y_j \otimes x_k \textcolor{red}{ - x_j \circ (x \circ (x_i \circ y_k)) \otimes y_i \otimes y_j \otimes x_k} \\
		&&\quad\quad + x \circ (x_j \circ (y_i \circ y_k)) \otimes x_i \otimes y_j \otimes x_k-\textcolor{red}{ x_j \circ (x \circ (y_i \circ y_k)) \otimes x_i \otimes y_j \otimes x_k} \\
		&&\quad\quad + \textcolor{green}{ x \circ (x_j \circ (x_k \circ x_i)) \otimes y_i \otimes y_j \otimes y_k}-\textcolor{blue}{ x_j \circ (x \circ (x_k \circ x_i)) \otimes y_i \otimes y_j \otimes y_k} \\
		&&\quad\quad + \textcolor{green}{x \circ (x_j \circ (x_k \circ y_i)) \otimes x_i \otimes y_j \otimes y_k} -\textcolor{blue}{ x_j \circ (x \circ (x_k \circ y_i)) \otimes x_i \otimes y_j \otimes y_k} \} \bigg) \\
		&&\;\overset{ }{=}-2 \mathfrak{S}.
	\end{eqnarray*}
	Hence, we have $\mathfrak{S} = -2\mathfrak{S}$ and $\mathfrak{S} = 0$.
	Therefore, Eq.~\eqref{CJBco:2} holds if and only if Eq.~\eqref{eq:jcs} holds.

	\ref{it:hr2}. For all $x,y \in A$, we have
	\begin{eqnarray*}
		&&r_{13} \circ \bigl( (L_{\circ}(x) L_{\circ}(y) \otimes \id \otimes \id )(r_{12} + r_{21}) \bigr) - (L_{\circ}(x) \otimes \id \otimes \id) \bigl( r_{13} \circ [(\id \otimes \id \otimes L_{\circ}(y))(r_{23} + r_{32})] \bigr) \notag\\
		&&\quad - \bigl( L_{\circ}(y) \otimes \id \otimes \id \bigr) \bigl( r_{13} \circ [(\id \otimes \id \otimes L_{\circ}(x))(r_{23} + r_{32})] \bigr) \\
		&&\overset{\hphantom{\eqref{CJBco:1}}}{=} \sum_{i,j} \Big\{ \uline{ (x \circ (y \circ y_i)) \circ x_j \otimes x_i \otimes y_j + (x \circ (y \circ x_i)) \circ x_j \otimes y_i \otimes y_j } \\
		&&\quad \uwave{ - x \circ x_i \otimes y_j \otimes (y \circ x_j) \circ y_i - x \circ x_i \otimes x_j \otimes (y \circ y_j) \circ y_i } \\
		&&\quad \dashuline{ - y \circ x_i \otimes y_j \otimes (x \circ x_j) \circ y_i - y \circ x_i \otimes x_j \otimes (x \circ y_j) \circ y_i} \Big\} \\
		&&\overset{\eqref{CJBco:1}}{=} \sum_{i,j} \Big\{ \uline{ (x \circ y_i) \circ x_j \otimes y \circ x_i \otimes y_j + (x \circ x_i) \circ x_j \otimes y \circ y_i \otimes y_j } \\
		&&\quad \uwave{- x \circ x_i \otimes y \circ y_j \otimes x_j \circ y_i - x \circ x_i \otimes y \circ x_j \otimes y_j \circ y_i } \\
		&&\quad \dashuline{ - y \circ x_i \otimes x \circ y_j \otimes x_j \circ y_i - y \circ x_i \otimes x \circ x_j \otimes y_j \circ y_i } \Big\} \\
		&&\overset{\eqref{CJBco:1}}{=} \sum_{i,j} \Big\{ \uline{ y_i \circ x_j \otimes y \circ (x \circ x_i) \otimes y_j + x_i \circ x_j \otimes y \circ (x \circ y_i) \otimes y_j} \\
		&&\quad - x \circ x_i \otimes y \circ y_j \otimes x_j \circ y_i - x \circ x_i \otimes y \circ x_j \otimes y_j \circ y_i \\
		&&\quad - y \circ x_i \otimes x \circ y_j \otimes x_j \circ y_i - y \circ x_i \otimes x \circ x_j \otimes y_j \circ y_i \Big\} \\
		&&\overset{\hphantom{\eqref{CJBco:1}}}{=} (\id \otimes L_{\circ}(y)L_{\circ}(x) \otimes \id)( r_{21} \circ r_{13} + r_{12} \circ r_{13}) \\
		&&\quad -(L_{\circ}(x) \otimes L_{\circ}(y) \otimes \id + L_{\circ}(y)\otimes L_{\circ}(x) \otimes \id)(r_{13} \circ r_{32} + r_{13} \circ r_{23}).
	\end{eqnarray*}
	Hence, Eq.~\eqref{eq:hou3.6} holds if and only if Eq.~\eqref{eq:jbcs} holds.
\end{proof}

\begin{definition}
	Let $(A, \circ)$ be a Jordan algebra.
	An element $r \in A \otimes A$ is called \textbf{invariant} if
	\begin{equation*}
		( \id \otimes L_{\circ}(x) - L_{\circ}(x) \otimes \id )(r) = 0, \;\; \forall x \in A.
	\end{equation*}
\end{definition}

\begin{theorem}\label{CJBforQuasiandTri}
	Let $(A, \circ)$ be a Jordan algebra and $r \in A \otimes A$.
	Define a linear map $\Delta: A \to A \otimes A$ by Eq.~\eqref{eq:jcbd}.
	If $r$ is a solution of the JYBE and the symmetric part of $r$ is invariant,
	then $(A, \circ, \Delta)$ is a Jordan bialgebra, called a \textbf{quasi-triangular Jordan D-bialgebra}.
	In particular, if $r$ is a skew-symmetric solution of the JYBE, then $(A, \circ, \Delta)$ is a Jordan bialgebra, called a \textbf{triangular Jordan D-bialgebra}.
\end{theorem}
\begin{proof}
	It follows from Proposition~\ref{prop:cobdppb} and Lemma~\ref{lem:sfh}.
\end{proof}

From now on, a quasi-triangular Jordan D-bialgebra is denoted by $(A, \circ, \Delta_r)$.

\begin{corollary}\label{coro:trr}
	Let $(A, \circ)$ be a Jordan algebra and $r \in A \otimes A$.
	Then $r$ is a solution of the JYBE in $(A, \circ)$ if and only if $\tau(r)$ is. 
	Therefore, $(A, \circ, \Delta_r)$ is a quasi-triangular Jordan D-bialgebra if and only if $(A, \circ, \Delta_{\tau(r)})$ is.
\end{corollary}
\begin{proof} 
	Note that $\mathbf{J}( \tau(r)) = \sigma_{13}(\mathbf{J}(r))$, and that the symmetric part of $r$ and $\tau(r)$ coincide. 
	The desired conclusion then follows immediately.
\end{proof}

%%%%%%%%%%%%%%%%%%%%%%%%%%%%%%%%%%%%%%%%%%%%%%%%%%%%%%%%%%%%%%%%%%%%%%%%%%%%%%%%
\subsection{Operators forms of the Jordan Yang-Baxter equations}

Let $A$ be a vector space.
Any $r \in A \otimes A$ can be identified as maps $r_{+}, r_{-}: A^* \to A$, defined respectively by
\begin{equation*}
	\langle r_{+}(a^*), b^*\rangle = -\langle a^*, r_{-}(b^*)\rangle = \langle r, a^* \otimes b^*\rangle, \;\; \forall a^*, b^* \in A^*.
\end{equation*}
Clearly, $-r_{-} = r_{+}^* = (\tau(r))_{+}$.
Note that $r$ is called the \textbf{2-tensor form} of a linear map $\varphi: A^* \to A$ if $r_{+} = \varphi$.
Moreover, the multiplication on $A^*$ defined by Eqs.~\eqref{eq:jcbd} (as the dual) is given by
\begin{equation}
	a^* \cdot_r b^* = L_{\circ}^*(r_{+}(a^*)) b^* + L_{\circ}^*(r_{-}(b^*)) a^*, \;\; \forall a^*, b^* \in A^*. \label{eq:pcbdr}
\end{equation}
Writing $r$ as $r = \Lambda + \Theta$ with $\Lambda \in \mathrm{Alt}^2(A)$ and $\Theta \in \mathrm{Sym}^2(A)$, i.e.,
\begin{equation*}
	\Lambda = \frac{r - \tau(r)}{2}, \;\; \Theta = \frac{r + \tau(r)}{2}
\end{equation*}
Immediately, $\tau(r) = -\Lambda + \Theta$ and 
\begin{equation*}
	\Lambda_{+} = \Lambda_{-} = \frac{r_{+} + r_{-}}{2}, \;\; \Theta_{+} = -\Theta_{-} = \frac{r_{+} - r_{-}}{2}.
\end{equation*}
That is, $\Lambda_{+}$ and $\Theta_{+}$ are the symmetrizer and antisymmetrizer of $r_{\pm}$, respectively.

\begin{lemma}\label{lem:sinvb}
	Let $(A, \circ)$ be a Jordan algebra and $\Theta \in A \otimes A$ be symmetric.
	Then the following conditions are equivalent.
	\begin{enumerate}[label=(\roman*)]
		\item\label{pbinv:1}
		      $\Theta$ is invariant.

		\item\label{pbinv:2}
		      $\Theta_{+}:  A^* \to A$ is a homomorphism of representations from $(A^*; L_{\circ}^*)$ to $(A; L_{\circ})$, i.e.,
		      \begin{align}
			      \Theta_{+}(L_{\circ}^*(x) a^*) = x \circ \Theta_{+}(a^*), \;\; \forall x \in A, a^* \in A^*. \label{eq:adi1}
		      \end{align}

		\item\label{pbinv:3}
		      $\Theta_{+}: A^* \to A$ is balanced associated to $(A^*; L_{\circ}^*)$, i.e.,
		      \begin{equation}
			      L_{\circ}^*(\Theta_+(a^*))b^* = L_{\circ}^*(\Theta_+(b^*))a^*, \;\; \forall a^*,b^* \in A^*. \label{eq:sadi1}
		      \end{equation}
	\end{enumerate}
\end{lemma}
\begin{proof}
	\ref{pbinv:1} $\Longleftrightarrow$ \ref{pbinv:2}: For all $x \in A$ and $a^*, b^* \in A^*$, we have
	\begin{align*}
		\langle \Theta_{+}(L_{\circ}^*(x) a^*) - x \circ \Theta_{+}(a^*), b^*\rangle & = \langle \Theta, L_{\circ}^*(x) a^* \otimes b^* - a^* \otimes L_{\circ}^*(x) b^* \rangle \\
		& = \langle (L_{\circ}(x) \otimes \id - \id \otimes L_{\circ}(x))\Theta, a^* \otimes b^* \rangle.
	\end{align*}
	That is, $\Theta$ is invariant if and only if Eq.~\eqref{eq:adi1} holds.

	\ref{pbinv:2} $\Longleftrightarrow $ \ref{pbinv:3}: For all $x \in A$ and $a^*, b^* \in A^*$, we have
	\begin{align*}
		\langle L_{\circ}^*(\Theta_{+}(a^*)) b^* - L_{\circ}^*(\Theta_{+}(b^*)) a^*, x \rangle & = \langle b^*, \Theta_{+}(a^*) \circ x \rangle - \langle a^*, \Theta_{+}(b^*) \circ x\rangle \\
		& = \langle b^*, x \circ \Theta_{+}(a^*) - \Theta_{+}(L_{\circ}^*(x)a^*) \rangle,
	\end{align*}
	which show that Eq.~\eqref{eq:adi1} holds if and only if Eq.~\eqref{eq:sadi1} holds.
\end{proof}

\begin{theorem}\label{thm:lreq}
	Let $(A, \circ)$ be a Jordan algebra and $r \in A \otimes A$.
	Writing $r$ as $r = \Lambda + \Theta$ with $\Lambda \in \mathrm{Alt}^2(A)$ and $\Theta \in \mathrm{Sym}^2(A)$.
	Suppose that $\Theta$ is invariant.
	Then the following conditions are equivalent.
	\begin{enumerate}[label=(\roman*)]
			\item\label{it:reqd1}
		      $r$ is a solution of the JYBE in $(A, \circ)$.
		\item\label{it:reqd2}
		      $\Lambda_{+}$ is an extended relative Rota-Baxter operator with extension $\Theta_{+}$ on $(A, \circ)$ associated to $(A^*; L_{\circ}^*)$.

		\item\label{it:reqd3}
		      $(A^*, \cdot_r)$ is a Jordan algebra, and $r_{+}$ (resp. $r_{-}$) is a homomorphism of Jordan algebras between $(A^*, \cdot_{r})$ and $(A, \circ)$, where $\cdot_{r}$ is defined by Eq.~\eqref{eq:pcbdr}.

		\item\label{it:reqd4}
		      $r_{+}$ (resp. $r_{-}$) is a relative Rota-Baxter operator of weight $1$ (resp. $-1$) on $(A, \circ)$ associated to the $A$-module Jordan algebra $((A^*, \cdot_{S}); L_{\circ}^*)$, where $\cdot_{S}$ is defined by
		      \begin{equation}
		      	a^* \cdot_{S} b^*  = - 2 L_{\circ}^*(\Theta_{+}(a^*)) b^*, \;\; a^*, b^* \in A^*. \label{eq:dmd}
		      \end{equation}

	\end{enumerate}
\end{theorem}
\begin{proof}\ref{it:reqd1} $\Longleftrightarrow$ \ref{it:reqd2}: 	Let $a^*, b^*, c^* \in A^*$. 
	Note that $\Theta_{+}: A^* \to A$ is a balanced homomorphism of representations from $(A^*; L_{\circ}^*)$ to $(A; L_{\circ})$, we have
	\begin{align*}
		 & r_{+}(a^{*}) \circ r_{+}(b^{*}) - r_{+}(L_{\circ}^*(r_{+}(a^{*})) b^{*}+L_{\circ}^{*}(r_{-}(b^{*})) a^{*})                                                                                                    \\
		 \overset{\hphantom{\eqref{eq:adi1}}}{=}\; & (\Lambda_{+}+\Theta_{+})(a^{*}) \circ (\Lambda_{+}+\Theta_{+})(b^{*}) - (\Lambda_{+}+\Theta_{+})(L_{\circ}^{*}((\Lambda_{+}+\Theta_{+})(a^{*})) b^{*}+L_{\circ}^{*}((\Lambda_{+}-\Theta_{+})(b^{*})) a^{*}) \\
		 \overset{\eqref{eq:sadi1}}{=}\; &  (\Lambda_{+}+\Theta_{+})(a^{*}) \circ (\Lambda_{+}+\Theta_{+})(b^{*}) - (\Lambda_{+}+\Theta_{+})(L_{\circ}^{*}(\Lambda_{+}(a^{*})) b^{*}+L_{\circ}^{*}(\Lambda_{+}(b^{*})) a^{*})                           \\
		 \overset{\hphantom{\eqref{eq:adi1}}}{=}\; & \Lambda_{+}(a^*) \circ \Lambda_{+}(b^*) - \Lambda_{+}\big(L_{\circ}^*(\Lambda_{+}(a^*))b^* + L_{\circ}^*(\Lambda_{+}(b^*))a^*\big) + \Theta_{+}(a^*) \circ \Theta_{+}(b^*)                                  \\
		 & + \Theta_{+}(a^*) \circ \Lambda_{+}(b^*) + \Lambda_{+}(a^*) \circ \Theta_{+}(b^*) - \Theta_{+}(L_{\circ}^{*}(\Lambda_{+}(x^{*})) y^{*}+L_{\circ}^{*}(\Lambda_{+}(y^{*})) x^{*})                         \\
		 \overset{\eqref{eq:adi1}}{=}\;&  \Lambda_{+}(a^*) \circ \Lambda_{+}(b^*) - \Lambda_{+}\big(L_{\circ}^*(\Lambda_{+}(a^*))b^* + L_{\circ}^*(\Lambda_{+}(b^*))a^*\big) + \Theta_{+}(a^*) \circ \Theta_{+}(b^*).
	\end{align*}
	Therefore,
	\begin{align*}
		& \langle \Lambda_{+}(a^*) \circ \Lambda_{+}(b^*) - \Lambda_{+}\big(L_{\circ}^*(\Lambda_{+}(a^*))b^* + L_{\circ}^*(\Lambda_{+}(b^*))a^*\big) + \Theta_{+}(a^*) \circ \Theta_{+}(b^*), c^* \rangle \\
		=\; & \langle r_{+}(x^{*}) \circ r_{+}(y^{*}) - r_{+}(L_{\circ}^{*}(r_{+}(x^{*})) y^{*}+L_{\circ}^{*}(r_{-}(y^{*})) x^{*}) , c^*\rangle \\
		=\; & \langle r_{12} \circ r_{13} - r_{12} \circ r_{23} + r_{13} \circ r_{23}, a^* \otimes b^* \otimes c^*\rangle.
	\end{align*}
	Hence, $r$ is a solution of the JYBE in $(A, \circ)$ if and only if $\Lambda_{+}: A^* \to A$ is an extended relative Rota-Baxter operator with extension $\Theta_{+}$ on $(A, \circ)$ associated to $(A^*; L_{\circ}^*)$.

\ref{it:reqd2} $\Longleftrightarrow$ \ref{it:reqd3} $\Longleftrightarrow$ \ref{it:reqd4}:  By Lemma~\ref{lem:sinvb}, $\Theta_{+}: A^* \to A$ is a balanced homomorphism of representations from $(A^*; L_{\circ}^*)$ to $(A; L_{\circ})$.
	Moreover, for all $a^*, b^* \in A^*$, we have
	\begin{align*}
		& L_{\circ}^*(\Lambda_{+}(a^*))b^* + L_{\circ}^*(\Lambda_{+}(b^*))a^* = L_{\circ}^*((r_{+}-\Theta_{+})(a^*))b^* + L_{\circ}^*((r_{-}+\Theta_{+})(b^*))a^* \\
		\overset{\eqref{eq:sadi1}}{=}\; & L_{\circ}^*(r_{+}(a^*)) b^* + L_{\circ}^*(r_{-}(b^*)) a^* = a^* \cdot_{r} b^*.
	\end{align*}
	The desired equivalence now follows immediately from Theorem~\ref{thm:opm}.
\end{proof}

\begin{remark}
	In particular, if $r$ is skew-symmetric, i.e., $\Theta=0$, then $r$ is a solution of the JYBE in $(A, \circ)$ if and only if $r_{+}$ is a relative Rota-Baxter operator of weight 0 on $(A, \circ)$ associated to $(A^*; L_{\circ}^*)$, as shown in~\cite[Proposition~3.1]{hou2013pre} in terms of $\mathcal{O}$-operators.
\end{remark}

%%%%%%%%%%%%%%%%%%%%%%%%%%%%%%%%%%%%%%%%%%%%%%%%%%%%%%%%%%%%%%%%%%%%%%%%%%%%%%%%
\subsection{Extended relative Rota-Baxter operators on quadratic Jordan algebras}

\begin{definition}
	A \textbf{quadratic Jordan algebra} is a triple $(A, \circ, \mathcal{B})$, where $(A, \circ)$ is a Jordan algebra and $\mathcal{B}$ is a nondegenerate symmetric bilinear form on $A$, which is \textbf{invariant} in the sense:
	\begin{equation*}
		\mathcal{B}(x \circ y, z)=\mathcal{B}(x, y \circ z), \;\; \forall x, y, z \in A.
	\end{equation*}
\end{definition}

\begin{proposition}\label{GouzaoQJA}
	Let $(A, \circ)$ be a Jordan algebra, and $(A \ltimes_{L_{\circ}^{*}} A^{*}, \bullet)$ be the semi-direct product of $(A, \circ)$ and $(A^*; L_{\circ}^*)$.
	Then $(A \ltimes_{L_{\circ}^{*}} A^{*}, \bullet, \mathcal{B}_d )$ is a quadratic Jordan algebra, where $\mathcal{B}_d$ is defined by
	\begin{equation}
		\mathcal{B}_{d}(x+a^{*}, y+b^{*})= \langle a^{*}, y \rangle + \langle b^{*}, x \rangle, \;\; \forall x, y \in A, \; a^{*}, b^{*} \in A^{*}. \label{eq:abm}
	\end{equation}
\end{proposition}
\begin{proof}
	For all $x, y,z \in A$ and $a^{*}, b^{*},c^{*}\in A^{*}$, we have
	\begin{align*}
		 & \mathcal{B}_{d}((x+a^{*}) \bullet (y+b^{*}),z + c^{*}) = \mathcal{B}_{d}(x \circ y + L_{\circ}^{*}(x) b^{*}+ L_{\circ}^{*}(y) a^{*},z + c^{*})                                                                                   \\
		 & =\langle c^{*}, x \circ y \rangle + \langle L_{\circ}^{*}(x) b^{*}+ L_{\circ}^{*}(y) a^{*}, z \rangle  =\langle L_{\circ}^{*}(y)c^{*}, x \rangle + \langle L_{\circ}^{*}(z)b^{*} , x \rangle + \langle a^{*}, z \circ y \rangle \\
		 & =\mathcal{B}_{d}( x+a^{*}, (y+b^{*})\bullet (z + c^{*})).
	\end{align*}
	Note that $\mathcal{B}_{d}$ is nondegenerate and symmetric.
	Hence, $(A \ltimes_{L_{\circ}^{*}} A^{*}, \bullet, \mathcal{B}_d )$ is a quadratic Jordan algebra.
	The proof is complete.
\end{proof}

\begin{example}\label{ex:Jordan3}
	Let $(A, \circ)$ be the 2-dimensional Jordan algebra in Example~\ref{ex:JordanXY} (a) with the basis $\{e_{1}, e_{2}\}$. 
	Let $\{e_1^*, e_2^*\}$ be the dual basis of $\{e_1, e_2\}$. 
	Then there is a 4-dimensional Jordan $(A \ltimes_{L_{\circ}^{*}} A^*, \bullet)$, whose nonzero product of $\bullet$ is explicitly given by 
	\begin{equation*}
		e_1 \bullet e_1 = e_1, \quad
		e_1 \bullet e_2 = \frac{1}{2} e_2 = e_2 \bullet e_1, \quad
		e_1 \bullet e_1^* =e_1^* = e_1^* \bullet e_1, \quad
		e_1 \bullet e_2^* =\frac{1}{2}e_2^* = e_2^* \bullet e_1.
	\end{equation*}
	By Proposition~\ref{GouzaoQJA}, $(A \ltimes_{L_{\circ}^{*}} A^{*}, \bullet, \mathcal{B}_d )$ is a quadratic Jordan algebra, where $\mathcal{B}_d$ is given by
	\begin{align*}
		\mathcal{B}_d(e_i, e_j) = \mathcal{B}_d(e_i^*, e_j^*) = 0, \quad \mathcal{B}_d(e_i^*, e_j) = \mathcal{B}_d(e_j, e_i^*) = \delta_{ij}, \quad i, j=1, 2.
	\end{align*}
\end{example}

Let $A$ be a vector space equipped with a nondegenerate bilinear form $\mathcal{B}$.
Denote by $I_{\mathcal{B}}: A^* \to A$ the induced linear isomorphism defined by
\begin{equation*}
	\langle I_{\mathcal{B}}^{-1}(x), y\rangle := \mathcal{B}(x, y), \quad \forall x,y \in A.
\end{equation*}

\begin{proposition}\label{dualdj:x}
	Let $(A, \circ)$ be a Jordan algebra.
	If there is a nondegenerate symmetric invariant bilinear form $\mathcal{B}$ such that $(A, \circ,\mathcal{B})$ is a quadratic Jordan algebra, then $(A; L_{\circ})$ and $(A^*; L_{\circ}^*)$ are equivalent representations of $(A, \circ)$. Conversely, if $(A; L_{\circ})$ and $(A^*; L_{\circ}^*)$ are equivalent representations of $(A, \circ)$, then there exists a nondegenerate invariant bilinear form $\mathcal{B}$ on $A$.
\end{proposition}
\begin{proof}
	For all $x, y, z \in A$, we have
	\begin{equation*}
		\langle I_\mathcal{B}^{-1}(L_{\circ}(x)y) , z\rangle = \mathcal{B}(x \circ y, z) = \mathcal{B}(y,x \circ z) = \langle I_\mathcal{B}^{-1}(y), x \circ z \rangle = \langle L^*_{\circ}(x)I_\mathcal{B}^{-1}(y), z\rangle,
	\end{equation*}
	Thus, $(A; L_{\circ})$ and $(A^*; L_{\circ}^*)$ are equivalent representations.
	
	Conversely, suppose $\varphi: A \to A^*$ is the linear isomorphism giving the equivalence between $(A; L_{\circ})$ and $(A^*; L_{\circ}^*)$.
	Define a bilinear form $\mathcal{B}$ on $A$ by
	\begin{equation*}
		\mathcal{B}(x, y) = \langle \varphi(x), y\rangle.
	\end{equation*}
	Then, by a similar argument as above, $\mathcal{B}$ is a nondegenerate invariant bilinear form on $A$.
\end{proof}

\begin{definition}
	Let $A$ be a vector space and $\mathcal{B}$ be a nondegenerate bilinear form on $A$.
	A linear map $T: A \to A$ is called \textbf{self-adjoint} (resp. \textbf{skew-adjoint}) with respect to $\mathcal{B}$ if for all $x, y \in A$
	\begin{equation*}
		\mathcal{B}(T(x), y) = \mathcal{B}(x, T(y)) \quad \text{(resp. $\mathcal{B}(T(x), y) = -\mathcal{B}(x, T(y))$)}.
	\end{equation*}
\end{definition}

\begin{lemma}\label{lem:sa2s}
	Let $(A, \circ, \mathcal{B})$ be a quadratic Jordan algebra and $I_\mathcal{B}$ be the induced linear isomorphism by $\mathcal{B}$. Let $T: A \to A$ be a linear map.
	Then $T$ is self-adjoint (resp. skew-adjoint) with respect to $\mathcal{B}$ if and only if the 2-tensor form of $T I_\mathcal{B}: A^* \to A$ is symmetric (resp. skew-symmetric).
\end{lemma}
\begin{proof}
	It suffices to prove the self-adjoint case, as the skew-adjoint case follows by the same argument.
	Let $r$ be the 2-tensor form of $T I_{\mathcal{B}}$, i.e., $r_{+} = T I_{\mathcal{B}}$.
	Then, for all $a^*, b^* \in A^*$, we have
	\begin{align*}
		& \langle r - \tau(r), a^* \otimes b^* \rangle = \langle b^*, r_{+}(a^*) \rangle - \langle a^*, r_{+}(b^*) \rangle = \mathcal{B}(I_{\mathcal{B}}(b^*), r_{+}(a^*)) - \mathcal{B}(I_{\mathcal{B}}(a^*), r_{+}(b^*)) \\
		=\; & \mathcal{B}(I_{\mathcal{B}}(b^*), T I_{\mathcal{B}}(a^*)) - \mathcal{B}(I_{\mathcal{B}}(a^*), T I_{\mathcal{B}}(b^*)) = \mathcal{B}(I_{\mathcal{B}}(b^*), T I_{\mathcal{B}}(a^*)) - \mathcal{B}(T I_{\mathcal{B}}(b^*), I_{\mathcal{B}}(a^*)).
	\end{align*}
	Hence, $r$ is symmetric if and only if $T$ is self-adjoint.
\end{proof}

\begin{lemma}\label{lem:eq4q}
	Let $(A, \circ, \mathcal{B})$ be a quadratic Jordan algebra and $I_\mathcal{B}$ be the induced linear isomorphism by $\mathcal{B}$.
	Suppose that $S: A \to A$ is a linear map that is self-adjoint with respect to $\mathcal{B}$.
	Then the following conditions are equivalent.
 \begin{enumerate}[label=(\roman*)]
		\item\label{it:eq41}
		      $S$ is balanced associated to $(A; L_{\circ})$.

		\item\label{it:eq42}
		      $S$ is a homomorphism of representations from $(A; L_{\circ})$ to $(A; L_{\circ})$.

		\item\label{it:eq43}
		      $P_S = S I_\mathcal{B}: A^* \to A$ is balanced associated to $(A^*; L_{\circ}^*)$.

		\item\label{it:eq44}
		      $P_S = S I_\mathcal{B}: A^* \to A$ is a homomorphism of representations from $(A^*; L_{\circ}^*)$ to $(A; L_{\circ})$.
	\end{enumerate}
\end{lemma}
\begin{proof}
	Let $a^*, b^* \in A^*$ and $x, y, z \in A$.
	Then the following identity
	\begin{align*}
		 & \mathcal{B}(S(x \circ y) - x \circ S(y), z) = \mathcal{B}(x \circ y, S(z)) - \mathcal{B}(x, S(y) \circ z) =  \mathcal{B}(x, y \circ S(z) - S(y) \circ z),
	\end{align*}
	imply the equivalence between \ref{it:eq41} and \ref{it:eq42}.
	Similarly, the following computation
	\begin{align*}
		& \langle L_{\circ}^*(P_S(a^*))b^* - L_{\circ}^*(P_S(b^*))a^*, x\rangle = \langle b^*, P_S(a^*) \circ x  \rangle - \langle a^*, P_S(b^*) \circ x \rangle \\
		=\;  & \mathcal{B}(I_\mathcal{B}(b^*), P_S(a^*) \circ x) - \mathcal{B}(I_\mathcal{B}(a^*), P_S(b^*) \circ x) = \mathcal{B}(SI_\mathcal{B}(a^*) \circ I_\mathcal{B}(b^*) - I_\mathcal{B}(a^*) \circ SI_\mathcal{B}(b^*), x) ,
	\end{align*}
	establishes the equivalence between \ref{it:eq41} and \ref{it:eq43}.
	Finally, identity
	\begin{align*}
		& \langle b^*, P_S(L_{\circ}^*(x)a^*) - x \circ P_S(a^*)\rangle = \mathcal{B}( I_\mathcal{B}(b^*), P_S(L_{\circ}^*(x)a^*)) - \mathcal{B}(I_\mathcal{B}(b^*), x \circ P_S(a^*)) \\
		=\; & \mathcal{B}(P_S(b^*), I_\mathcal{B}(L_{\circ}^*(x)a^*)) - \mathcal{B}(x \circ P_S(a^*), I_\mathcal{B}(b^*) ) = \langle a^*, x \circ P_S(b^*) \rangle - \mathcal{B}(x \circ P_S(a^*), I_\mathcal{B}(b^*) ) \\
		=\; & \mathcal{B}(I_\mathcal{B}(a^*), x \circ P_S(b^*)) - \mathcal{B}(x \circ P_S(a^*), I_\mathcal{B}(b^*) ) = \mathcal{B}(x, SI_\mathcal{B}(b^*) \circ I_\mathcal{B}(a^*)  - I_\mathcal{B}(b^*) \circ SI_\mathcal{B}(a^*) ),
	\end{align*}
	show the equivalence between \ref{it:eq41} and \ref{it:eq44}.
	The proof is complete.
\end{proof}

\begin{proposition}\label{prop:ajo2s}
	Let $(A, \circ, \mathcal{B})$ be a quadratic Jordan algebra, $I_\mathcal{B}$ be the induced linear isomorphism by $\mathcal{B}$, and $T, S: A \to A$ be linear maps.
	Then $T$ is an extended relative Rota-Baxter operator with extension $S$ associated to $(A; L_{\circ})$ if and only if $P_T=T I_\mathcal{B}: A^* \to A$ is an extended relative Rota-Baxter operator with extension $P_S = S I_\mathcal{B}: A^* \to A$ associated to $(A^*; L_{\circ}^*)$.
\end{proposition}
\begin{proof}
	For all $x, y \in A$ and $a^* \in A^*$, we have
	\begin{equation*}
		\mathcal{B}(I_\mathcal{B}(L_{\circ}^*(x)a^*), y) = \langle L_{\circ}^*(x)a^*, y\rangle = \langle a^*, x \circ y\rangle = \mathcal{B}(I_\mathcal{B}(a^*), x \circ y) = \mathcal{B}(x \circ I_\mathcal{B}(a^*), y).
	\end{equation*}
	That is, $I_\mathcal{B}(L_{\circ}^*(x)a^*) = x \circ I_\mathcal{B}(a^*)$.
	Therefore,
	\begin{align*}
		 & P_T(a^*) \circ P_T(b^*) - P_T\big(L_{\circ}^*(P_T(a^*))b^* + L_{\circ}^*(P_T(b^*))a^* \big) + P_S(a^*) \circ P_S(b^*)                                                                                            \\
		 & =TI_\mathcal{B}(a^*) \circ TI_\mathcal{B}(b^*) - TI_\mathcal{B}\big(L_{\circ}^*(TI_\mathcal{B}(a^*))b^* + L_{\circ}^*(TI_\mathcal{B}(b^*))a^* \big) + SI_\mathcal{B}(a^*) \circ SI_\mathcal{B}(b^*)              \\
		 & = TI_\mathcal{B}(a^*) \circ TI_\mathcal{B}(b^*) - T\big(L_{\circ}(TI_\mathcal{B}(a^*))I_\mathcal{B}(b^*) + L_{\circ}(TI_\mathcal{B}(b^*))I_\mathcal{B}(a^*) \big) + SI_\mathcal{B}(a^*) \circ SI_\mathcal{B}(b^*).
	\end{align*}
	Hence, $T$ is an extended relative Rota-Baxter operator with extension $S$ associated to $(A; L_{\circ})$ if and only if $P_T$ is an extended relative Rota-Baxter operator with extension $P_S$ associated to $(A^*; L_{\circ}^*)$.
\end{proof}

\begin{corollary}\label{coro:exoqs}
	Let $(A, \circ, \mathcal{B})$ be a quadratic Jordan algebra, $I_\mathcal{B}$ be the induced linear isomorphism by $\mathcal{B}$, and $T, S: A \to A$ be linear maps.
	Suppose that $S$ is balanced associated to $(A; L_{\circ})$ and self-adjoint with respect to $\mathcal{B}$, and $T$ is skew-adjoint with respect to $\mathcal{B}$.
	Then the 2-tensor form of $T I_\mathcal{B} + S I_\mathcal{B}$ (resp. $T I_\mathcal{B} - S I_\mathcal{B}$) is a solution of the JYBE in $(A, \circ)$ if and only if $T$ is an extended relative Rota-Baxter operator with extension $S$ associated to $(A; L_{\circ})$.
	In particular, the 2-tensor form of $T I_\mathcal{B}$ is a solution of the JYBE in $(A, \circ)$ if and only if $T$ is a Rota-Baxter operator of weight $0$.
\end{corollary}
\begin{proof}
	By Lemma~\ref{lem:sa2s}, the 2-tensor form of $T I_\mathcal{B}$ is skew-symmetric, while that of $S I_\mathcal{B}$ is symmetric.
	Since $S$ is balanced associated to $(A; L_{\circ})$, it follows from Lemmas~\ref{lem:sinvb} and \ref{lem:eq4q}   that the 2-tensor form of $S I_\mathcal{B}$ is invariant.
	Then, by Theorem~\ref{thm:lreq}, the 2-tensor form of $T I_\mathcal{B} + S I_\mathcal{B}$ (resp. $T I_\mathcal{B} - S I_\mathcal{B}$) is a solution of the JYBE in $(A, \circ)$ if and only if $T I_\mathcal{B}: A^* \to A$ is an extended relative Rota-Baxter operator with extension $S I_\mathcal{B}$ associated to $(A^*; L_{\circ}^*)$.
	Therefore, Proposition~\ref{prop:ajo2s} implies that the 2-tensor form of $T I_\mathcal{B} + S I_\mathcal{B}$ (resp. $T I_\mathcal{B} - S I_\mathcal{B}$) is a solution of the JYBE in $(A, \circ)$ if and only if $T$ is an extended relative Rota-Baxter operator with extension $S$ associated to $(A; L_{\circ})$.
	Finally, the particular case is obtained by taking $S = 0$.
\end{proof}

\begin{corollary}\label{coro:rexoqs}
	Let $(A, \circ, \mathcal{B})$ be a quadratic Jordan algebra, $I_\mathcal{B}: A^* \to A$ be the induced linear isomorphism by $\mathcal{B}$ and $r \in A \otimes A$.
	Writing $r$ as $r = \Lambda + \Theta$ with $\Lambda \in \mathrm{Alt}^2(A)$ and $\Theta \in \mathrm{Sym}^2(A)$.
	Suppose that $\Theta$ is invariant.
	Then $r$ is a solution of the JYBE in $(A, \circ)$ if and only if $\Lambda_{+}I_\mathcal{B}^{-1}: A \to A$ is an extended relative Rota-Baxter operator with extension $\Theta_{+}I_\mathcal{B}^{-1}: A \to A$ associated to $(A; L_{\circ})$.
	In particular, $\Lambda$ is a solution of the JYBE in $(A, \circ)$ if and only if $\Lambda_{+}I_\mathcal{B}^{-1}: A \to A$ is a Rota-Baxter operator of weight 0.
\end{corollary}
\begin{proof}
	By Lemma~\ref{lem:sa2s}, $\Lambda_{+}I_\mathcal{B}^{-1}$ is skew-adjoint with respect to $\mathcal{B}$, and $\Theta_{+}I_\mathcal{B}^{-1}$ is self-adjoint with respect to $\mathcal{B}$.
	Moreover, since $\Theta$ is invariant, it follows from Lemma~\ref{lem:sinvb} and \ref{lem:eq4q} that $\Theta_{+}I_\mathcal{B}^{-1}$ is balanced associated to $(A; L_{\circ})$.
	The desired conclusions now follow from Corollary~\ref{coro:exoqs}.
\end{proof}

\begin{example}
	Let $(A, \circ)$ be the 2-dimensional Jordan algebra with basis $\{e_1, e_2\}$ whose non-zero product is defined by
	\begin{align*}
		e _1 \circ e_1 & = e_1, \quad e_1 \circ e_2 = e_2 = e_2 \circ e_1.
	\end{align*}
	Define a linear map $T: A \to A$ by
	\begin{equation*}
		T(e_1) =   e_1, \;\; T(e_2) = -   e_2.
	\end{equation*}
	Then $T$ is an extended relative Rota-Baxter operator with extension $S=\id$ associated to $(A; L_{\circ})$.
	Let $\mathcal{B}$ be the bilinear form defined by
	\begin{equation*}
		\mathcal{B}(e_1, e_1) = \mathcal{B}(e_2, e_2) = 0, \;\;
		\mathcal{B}(e_1, e_2) = \mathcal{B}(e_2, e_1) = 1.
	\end{equation*}
	Then $(A, \circ, \mathcal{B})$ is a quadratic Jordan algebra.
	Moreover $T$ is skew-adjoint with respect to $\mathcal{B}$.
	Thus by Corollary \ref{coro:exoqs}, the 2-tensor form of $T I_\mathcal{B} \pm S I_\mathcal{B}$:
	\begin{equation*}
		2 e_1 \otimes e_2, \quad 2e_2 \otimes e_1 
	\end{equation*}
	are solutions of the JYBE in $(A, \circ)$.
\end{example}

%%%%%%%%%%%%%%%%%%%%%%%%%%%%%%%%%%%%%%%%%%%%%%%%%%%%%%%%%%%%%%%%%%%%%%%%%%%%%%%%
\subsection{Extended relative Rota-Baxter operators on semi-direct product}

Let $T: V \to A$ be a linear map.
Through the identification $\Hom(V, A) \cong A \otimes V^* \subset (A \oplus V^*) \otimes (A \oplus V^*)$, we may regard $T$ as an element $r^{T} \in (A \oplus V^*) \otimes (A \oplus V^*)$ of the tensor product.
More precisely, for a basis $\{e_1, \cdots, e_n\}$ of $V$ with the dual basis $\{e_1^*, \cdots, e_n^*\}$, we have $r^T = \sum_{i} T(e_i) \otimes e_i^*$.

\begin{lemma}\label{lem:b2bsm}
	Let $(A, \circ)$ be a Jordan algebra, $(V; \rho)$ be a representation of $(A, \circ)$, and $(\mathfrak{A}=A \ltimes_{\rho^*} V^*, \bullet)$ be the semi-direct product of $(A, \circ)$ and $(V^*; \rho^*)$.
	Let $S: V \to A$ be a linear map.
	Then $\mathcal{S} := r^{S}_{+} - r^{S}_{-}: \mathfrak{A}^* \to \mathfrak{A}$ is a balanced homomorphism of representations from $(\mathfrak{A}^*; L_{\bullet}^*)$ to $(\mathfrak{A}; L_{\bullet})$ if and only if $S$ is a balanced homomorphism of representations from $(V; \rho)$ to $(A; L_{\circ})$.
\end{lemma}
\begin{proof}
	Note that $r^S + \tau(r^S)$ is symmetric and $(r^S + \tau(r^S))_{+} = r^{S}_{+} - r^{S}_{-} = \mathcal{S}$. By Lemma~\ref{lem:sinvb}, $\mathcal{S}: \mathfrak{A}^* \to \mathfrak{A}$ is a balanced homomorphism of representations from $(\mathfrak{A}^*; L_{\bullet}^*)$ to $(\mathfrak{A}; L_{\bullet})$ if and only if it is balanced associated to $(\mathfrak{A}^*; L_{\bullet}^*)$.
	It then suffices to prove that $\mathcal{S}$ is balanced associated to $(\mathfrak{A}^*; L_{\bullet}^*)$ if and only if $S$ is a balanced homomorphism of representations from $(V; \rho)$ to $(A; L_{\circ})$.
	For all $a^* \in A^*$ and $u \in V$, we have
	\begin{equation*}
		\mathcal{S}(a^* + u) = S(u) + S^*(a^*).
	\end{equation*}
	Let $a^*, b^* \in A^*$, $x \in A$, $\xi^* \in V^*$ and $u, v \in V$, we have
	\begin{align*}
		& \langle -L_{\bullet}^*((r^{S}_{+} - r^{S}_{-})(a^* + u)) (b^* + v) + L_{\bullet}^*((r^{S}_{+} - r^{S}_{-})(b^* + v)) (a^* + u), x + \xi^*\rangle \\
		=\; & \langle -L_{\bullet}^*(S(u) + S^*(a^*)) (b^* + v) + L_{\bullet}^*(S(v) + S^*(b^*)) (a^* + u), x + \xi^* \rangle                                \\
		=\; & -\langle b^* + v, (S(u) + S^*(a^*)) \bullet (x + \xi^*)\rangle + \langle a^* + u, (S(v) + S^*(b^*)) \bullet (x + \xi^*)  \rangle \\
		=\; & \langle a^*, S(v) \circ x \rangle + \langle \rho^*(S(v))\xi^* + \rho^*(x)S^*(b^*), u \rangle -\langle b^*, S(u) \circ x \rangle - \langle \rho^*(S(u))\xi^* + \rho^*(x)S^*(a^*), v \rangle \\
		=\; & \langle a^*, S(v) \circ x - S(\rho(x)v)\rangle - \langle b^*, S(u) \circ x - S(
		\rho(x)u) \rangle + \langle \xi^*, \rho(S(v))u - \rho(S(u))v \rangle.
	\end{align*}
	Therefore, $\mathcal{S}: \mathfrak{A}^* \to \mathfrak{A}$ is balanced associated to $(\mathfrak{A}^*; L_{\bullet}^*)$ if and only if $S$ is a balanced homomorphism of representations from $(V; \rho)$ to $(A; L_{\circ})$.
	The proof is complete.
\end{proof}

\begin{theorem}\label{thm:eo2smp}
	Let $(A, \circ)$ be a Jordan algebra, $(V; \rho)$ be a representation of $(A, \circ)$, and $(\mathfrak{A}=A \ltimes_{\rho^*} V^*, \bullet)$ be the semi-direct product of $(A, \circ)$ and $(V^*; \rho^*)$.
	Let $T, S: V \to A$ be linear maps.
	Suppose that $S$ is a homomorphism of representations from $(V; \rho)$ to $(A; L_{\circ})$.
	Then $\mathcal{T} := r_{+}^{T} + r_{-}^{T}: \mathfrak{A}^* \to \mathfrak{A}$ is an extended relative Rota-Baxter operator with extension $\mathcal{S} := r_{+}^{S} - r_{-}^{S}$ on $(\mathfrak{A}, \bullet)$ associated to $(\mathfrak{A}^*; L_{\bullet}^*)$ if and only if $T: V \to A$ is an extended relative Rota-Baxter operator of weight 0 with extension $S$ associated to $(V; \rho)$.
\end{theorem}
\begin{proof}
	For all $a^* \in A^*$ and $v \in V$, we have
	\begin{equation*}
		\mathcal{T}(a^* + u) =  -T(u) + T^*(a^*).
	\end{equation*}
	For all $a^*, b^*, c^* \in A^*$ and $u, v, w \in V$, we have
	\begin{align*}
		&\langle \mathcal{T}(a^* + u) \bullet \mathcal{T}(b^* + v) + \mathcal{S}(a^*+u) \bullet \mathcal{S}(b^*+v), c^*+ w\rangle                             \\
		& - \langle \mathcal{T}\big( L_{\bullet}^*(\mathcal{T}(a^* + u))(b^*+v) + L_{\bullet}^*(\mathcal{T}(b^* + v))(a^*+u)\big), c^*+ w\rangle \\
		\overset{\hphantom{\eqref{eq:ainv}}}{=}\; & \langle (-T(u) + T^*(a^*)) \bullet ( -T(v) + T^*(b^*)) + (S(u) + S^*(a^*)) \bullet (S(v) + S^*(b^*)), c^*+ w\rangle \\
		& + \langle L_{\bullet}^*(-T(u) + T^*(a^*))(b^*+v) + L_{\bullet}^*(-T(v) + T^*(b^*))(a^*+u), \mathcal{T}(c^*+ w) \rangle \\
		\overset{\hphantom{\eqref{eq:ainv}}}{=}\; & \langle c^*, T(u) \circ T(v) + S(u) \circ S(v) \rangle \\
		& - \langle \rho^*(T(u))T^*(b^*) + \rho^*(T(v))T^*(a^*) - \rho^*(S(u))S^*(b^*) - \rho^*(S(v))S^*(a^*), w\rangle                                 \\
		& + \langle L_{\bullet}^*(-T(u) + T^*(a^*) )(b^*+v) + L_{\bullet}^*(-T(v) + T^*(b^*) )(a^*+u), \mathcal{T}(c^* + w)\rangle                    \\
		\overset{\hphantom{\eqref{eq:ainv}}}{=}\; &\langle c^*, T(u) \circ T(v) + S(u) \circ S(v) \rangle \\
		& - \langle b^*, T(\rho(T(u))w) - S(\rho(S(u))w)\rangle - \langle a^*, T(\rho(T(v))w) - S(\rho (S(v))w) \rangle                                  \\
		& + \langle b^* + v, (-T(u) + T^*(a^*)) \bullet (-T(w) + T^*(c^*)) \rangle \\
		& + \langle a^* + u, (-T(v) + T^*(b^*)) \bullet (-T(w) + T^*(c^*))  \rangle \\
		\overset{\hphantom{\eqref{eq:ainv}}}{=}\; & \langle c^*, T(u) \circ T(v) + S(u) \circ S(v) \rangle \\
		& - \langle b^*, T(\rho(T(u))w) - S(\rho(S(u))w)\rangle - \langle a^*, T(\rho(T(v))w) - S(\rho(S(v))w) \rangle \\
		& + \langle b^* + v, T(u) \circ T(w) - \rho^*(T(u))(T^*(c^*)) - \rho^*(T(w))(T^*(a^*)) \rangle \\
		& + \langle a^* + u, T(v) \circ T(w)  - \rho^*(T(v))(T^*(c^*)) - \rho^*(T(w))(T^*(b^*))  \rangle \\
		\overset{\hphantom{\eqref{eq:ainv}}}{=}\; & \langle c^*, T(u) \circ T(v) - T(\rho(T(u))v) - T(\rho(T(v))u) + S(u) \circ S(v) \rangle \\
		& - \langle b^*, T(\rho(T(u))w)   - S(\rho(S(u))w) - T(u) \circ T(w) + T(\rho(T(w))u) \rangle \\
		& - \langle a^*, T(\rho(T(v))w) - S(\rho(S(v))w)  + T(\rho(T(w))v) - T(v) \circ T(w) \rangle \\
		\overset{\eqref{eq:ainv}}{=}\; & \langle c^*, T(u) \circ T(v) - T(\rho(T(u))v) - T(\rho(T(v))u) + S(u) \circ S(v) \rangle \\
		& + \langle b^*, T(u) \circ T(w)  - T(\rho(T(u))w) - T(\rho(T(w))u)  + S(u) \circ S(w) \rangle \\
		& + \langle a^*,  T(v) \circ T(w) - T(\rho(T(v))w) - T(\rho(T(w))v)  + S(v) \circ S(w) \rangle.
	\end{align*}
	Therefore, $\mathcal{T}$ is an extended relative Rota-Baxter operator with extension $\mathcal{S}$ on $(\mathfrak{A}, \bullet)$ associated $(\mathfrak{A}^*; L_{\bullet}^*)$ if and only if $T: V \to A$ is an extended relative Rota-Baxter operator with extension $S$ on $(A, \circ)$ associated to $(V; \rho)$.
\end{proof}

\begin{corollary}\label{cor:tsc}
	Let $(A, \circ)$ be a Jordan algebra, $(V; \rho)$ be a representation of $(A, \circ)$, and $(\mathfrak{A} = A \ltimes_{\rho^*} V^*, \bullet)$ be the semi-direct product of $(A, \circ)$ and $(V^*; \rho^*)$.
	Suppose that $S: V \to A$ is a balanced homomorphism of representations from $(V; \rho)$ to $(A; L_{\circ})$.
	Then $T: V \to A$ is an extended relative Rota-Baxter operator with extension $S$ on $(A, \circ)$ associated to $(V; \rho)$ if and only if $(r^T - \tau(r^T)) + (r^S + \tau(r^S))$ (resp. $(r^T - \tau(r^T)) - (r^S + \tau(r^S))$) is a solution of the JYBE in $(A \ltimes_{\rho^*} V^*, \bullet)$.
	In particular, $T$ is a relative Rota-Baxter operator of weight 0 associated to $(V; \rho)$ if and only if $r^T - \tau(r^T)$ is a skew-symmetric solution of the JYBE in $(A \ltimes_{\rho^*} V^*, \bullet)$.
\end{corollary}
\begin{proof}
	By Lemmas~\ref{lem:b2bsm} and \ref{lem:sinvb}, $r^S + \tau(r^S)$ is invariant.
	Applying Theorem~\ref{thm:lreq}, we get that $(r^T - \tau(r^T)) + (r^S + \tau(r^S))$ (resp. $(r^T - \tau(r^T)) - (r^S + \tau(r^S))$) is a solution of the JYBE in $(A \ltimes_{\rho^*} V^*, \bullet)$ if and only if $r_{+}^{T} + r_{-}^{T}: \mathfrak{A}^* \to \mathfrak{A}$ is an extended relative Rota-Baxter operator with extension $r_{+}^{S} - r_{-}^{S}$ on $(\mathfrak{A}, \bullet)$ associated to $(\mathfrak{A}^*; L_{\bullet}^*)$.
	By Theorem~\ref{thm:eo2smp}, the latter condition is equivalent to $T: V \to A$ being an extended relative Rota-Baxter operator with extension $S$ associated to $(V; \rho)$.
	Thus, we obtain the desired equivalence.
	Particular case follows by taking $S = 0$.
\end{proof}

\begin{corollary}\label{coro:eo2jy}
	Let $(A, \circ)$ be a Jordan algebra and  $(\mathfrak{A} = A \ltimes_{L_{\circ}^*} A^*, \bullet)$ be the semi-direct product of $(A, \circ)$ and $(A^*; L_{\circ}^*)$.
	Let $T: A \to A$ be a linear map.
	Then
 \begin{enumerate}[label=(\roman*)]
		\item\label{it:c1}
		      $T$ is an extended relative Rota-Baxter operator with extension $\id$ associated to $(A; L_{\circ})$ if and only if $r^T - \tau(r^T) + (r^\id + \tau(r^\id))$ (resp. $r^T - \tau(r^T) - (r^\id + \tau(r^\id))$) is a solution of the JYBE in $(A \ltimes_{L_{\circ}^*} A^*, \bullet)$.

		\item\label{it:c2}
		      $T$ is a Rota-Baxter operator of weight $\lambda$ if and only if $(r^T - \tau(r^T)) + \lambda r^\id$ (resp. $(r^T - \tau(r^T)) - \lambda \tau(r^\id)$)  is a solution of the JYBE in $(A \ltimes_{L_{\circ}^*} A^*, \bullet)$.
	\end{enumerate}
\end{corollary}
\begin{proof}
	 \ref{it:c1}.
	The conclusion follows from Corollary~\ref{cor:tsc} by taking $S = \id$.

	\ref{it:c2}.
	By Lemma~\ref{lem:rb2eo}, $T: A \to A$ is a Rota-Baxter operator of weight $\lambda$ if and only if $T + \frac{\lambda}{2}\id$ is an extended relative Rota-Baxter operator with extension $\frac{\lambda}{2}\id$ associated to $(A; L_{\circ})$.
	Therefore, the conclusion follows from Corollary~\ref{cor:tsc}.
\end{proof}

%%%%%%%%%%%%%%%%%%%%%%%%%%%%%%%%%%%%%%%%%%%%%%%%%%%%%%%%%%%%%%%%%%%%%%%%%%%%%%%%
%%%%%%%%%%%%%%%%%%%%%%%%%%%%%%%%%%%%%%%%%%%%%%%%%%%%%%%%%%%%%%%%%%%%%%%%%%%%%%%%
%%%%%%%%%%%%%%%%%%%%%%%%%%%%%%%%%%%%%%%%%%%%%%%%%%%%%%%%%%%%%%%%%%%%%%%%%%%%%%%%
\section{Factorizable Jordan D-bialgebras and quadratic Rota-Baxter Jordan algebras}\label{sec:quadrbj}
In this section, we introduce a special class of quasi-triangular Jordan D-bialgebras, called a factorizable Jordan D-bialgebra, and show that the Drinfeld classical double of any Jordan D-bialgebra naturally carries a factorizable Jordan D-bialgebra structure.
We also introduce quadratic Rota-Baxter Jordan algebras and establish a one-to-one correspondence between factorizable Jordan D-bialgebras and quadratic Rota-Baxter Jordan algebras of nonzero weights.
Moreover, we demonstrate that every quadratic Rota-Baxter Jordan algebra of zero weight gives rise to a triangular Jordan D-bialgebra. 

%%%%%%%%%%%%%%%%%%%%%%%%%%%%%%%%%%%%%%%%%%%%%%%%%%%%%%%%%%%%%%%%%%%%%%%%%%%%%%%%
\subsection{Factorizable Jordan D-bialgebras}

\begin{definition}
	A quasi-triangular Jordan D-bialgebra $(A,\circ, \Delta_{r})$ is called \textbf{factorizable} if $r_{+} - r_{-}: A^* \to A$ is a linear isomorphism of vector spaces.
\end{definition}

\begin{proposition}\label{prop:factorJba}
	Let $(A, \circ)$ be a Jordan algebra and $r \in A \otimes A$.
	Then $(A, \circ, \Delta_r)$ is a factorizable Jordan D-bialgebra if and only if $(A, \circ, \Delta_{\tau(r)})$ is a factorizable Jordan D-bialgebra.
\end{proposition}
\begin{proof}
	It follows from Corollary~\ref{coro:trr}.
\end{proof}

Consider the linear map
\begin{equation*}
	A^* \xrightarrow{\;\;\;\; r_{+} \oplus r_{-}\;\;\;\;} A \oplus A \xrightarrow{(a, b) \mapsto a + b} A.
\end{equation*}
The following result justifies the terminology of a factorizable Jordan D-bialgebra.
\begin{proposition}\label{prop:fpt}
	Let $(A, \circ, \Delta_{r})$ be a factorizable Jordan D-bialgebra.
	Then $\mathrm{Im}(r_{+} \oplus r_{-})$ is a Jordan subalgebra of the direct sum Jordan algebra $A \oplus A$,
	which is isomorphic to the Jordan algebra $(A^*, \cdot_{r})$, where $\cdot_{r}: A^* \otimes A^* \to A^*$ is defined by Eq.~\eqref{eq:pcbdr}.
	Moreover, any $x \in A$ has a unique decomposition
	$x = x_{+} - x_{-}$ with $(x_{+}, x_{-}) \in \mathrm{Im}(r_{+} \oplus r_{-})$.
\end{proposition}
\begin{proof}
	By Theorem~\ref{thm:lreq}, both $r_{+}$ and $r_{-}$ are Jordan algebra homomorphisms.
	Therefore, $\mathrm{Im}(r_{+} \oplus r_{-})$ is a Jordan subalgebra of the direct sum Jordan algebra $A \oplus A$.
	Since $I := r_{+} - r_{-}: A^* \to A$ is a linear isomorphism of vector spaces, it follows that the Jordan algebra $\mathrm{Im}(r_{+} \oplus r_{-})$ is isomorphic to the Jordan algebra $(A^*, \cdot_r)$.
	Moreover, for any $x \in A$, we have
	\begin{equation*}
		r_{+}I^{-1}(x) - r_{-}I^{-1}(x) = (r_{+} - r_{-}) I^{-1}(x)= x,
	\end{equation*}
	which shows that $x = x_{+} - x_{-}$ with $x_{+} = r_{+}I^{-1}(x)$ and $x_{-} = r_{-}I^{-1}(x)$.
	The uniqueness also follows from the fact that $I= r_{+} - r_{-}: A^* \to A$ is a linear isomorphism of vector spaces.
\end{proof}

Let $(A, \circ_A, \Delta)$ be an arbitrary Jordan D-bialgebra and $\circ_{A^*}: A^* \otimes A^* \to A^*$ be the linear dual of $\Delta$. 
Then the direct sum $\mathfrak{D}: = A \oplus A^*$ carries a Jordan algebra structure $\circ_{\mathfrak{D}}$ given by
\begin{equation*}
	(x+a^*) \circ_{\mathfrak{D}} (y+b^*) = x \circ_A y + L_{\circ_{A^*}}^{*}(a^*) y +L_{\circ_{A^*}}^{*}(b^*) x+L_{\circ_A}^{*}(x) b^*+L_{\circ_A}^{*}(y) a^* + a^* \circ_{A^*} b^*,
\end{equation*}
for all $x,y \in A$ and $a^*, b^* \in A^*$. 
We call $(\mathfrak{D}, \circ_{\mathfrak{D}})$ the \textbf{Drinfeld classical double} of the Jordan D-bialgebra $(A, \circ_A, \Delta)$. 
Moreover, $(\mathfrak{D}, \circ_{\mathfrak{D}}, \mathcal{B}_{d})$ is a quadratic Jordan algebra, where $\mathcal{B}_{d}$ is defined by Eq.~\eqref{eq:abm}~\cite{hou2022new}.

\begin{theorem}
	Let $(A, \circ_A, \Delta)$ be a Jordan D-bialgebra and $(\mathfrak{D}, \circ_{\mathfrak{D}})$ be the Drinfeld classical double of $(A, \circ_A, \Delta)$. 
	Let $\{e_{1}, e_{2}, \ldots, e_{n}\}$ be a basis of $A$ and $\{e_{1}^{*}, e_{2}^{*}, \ldots ,$ $e_{n}^{*}\}$ be the dual basis. 
	Set
	\begin{equation*}
		r=\sum_{i}^{n} e_{i} \otimes e_{i}^{*} \in A \otimes A^{*} \subset \mathfrak{D} \otimes \mathfrak{D}.
	\end{equation*}
	Then $\left(\mathfrak{D},\circ_{\mathfrak{D}},\Delta_{r} \right)$ is a factorizable Jordan D-bialgebra with $\Delta_{r}: \mathfrak{D} \to \mathfrak{D} \otimes \mathfrak{D}$ defined by Eq.~\eqref{eq:jcbd}.
\end{theorem}
\begin{proof}
	By \cite{hou2022new}, $\left(\mathfrak{D},\circ_{\mathfrak{D}}, \mathcal{B}_{d} \right)$ is a quadratic Jordan algebra.
	Let $I_{\mathcal{B}_{d}}: \mathfrak{D}^* \to \mathfrak{D}$ be the induced linear isomorphism by $\mathcal{B}_{d}$ and $T: \mathfrak{D} \to \mathfrak{D}$ be a linear map defined by
	\begin{equation*}
		T(x + a^*) = \frac{1}{2}(-x + a^*), \;\; \forall x \in A, \; a^* \in A^*.
	\end{equation*}
	Then, for all $x, y \in A$ and $a^*, b^* \in A^*$,
	\begin{equation*}
		\mathcal{B}_{d}(T(x + a^*), y + b^*) = -\frac{1}{2}\langle b^*, x\rangle + \frac{1}{2}\langle a^*, y\rangle = - \mathcal{B}_{d}(x + a^*, T(y + b^*)).
	\end{equation*}
	Thus, $T$ is skew-adjoint with respect to $\mathcal{B}_{d}$.
	Furthermore, $T$ is an extended relative Rota-Baxter operator with extension $S = \frac{1}{2}\id$ on $(\mathfrak{D}, \circ_{\mathfrak{D}})$ associated to $(\mathfrak{D}; L_{\circ_{\mathfrak{D}}})$.
	Therefore, by Corollary~\ref{coro:exoqs}, the 2-tensor form of $T I_{\mathcal{B}_d} +  SI_{\mathcal{B}_d}$  is a solution of the JYBE in $(\mathfrak{D},\circ_{\mathfrak{D}})$.
	Note that
	\begin{equation*}
		r_{+}(a^* + x) = a^* = (TI_{\mathcal{B}_d} + SI_{\mathcal{B}_d})(a^* + x).
	\end{equation*}
	Hence, $r$ is a solution of the JYBE in $(\mathfrak{D}, \circ_{\mathfrak{D}})$.
	Moreover, $r_{+} - r_{-} = I_{\mathcal{B}_d}$.
	Hence, $\left(\mathfrak{D},\circ_{\mathfrak{D}},\Delta_{r} \right)$ is a factorizable Jordan D-bialgebra.
\end{proof}

%%%%%%%%%%%%%%%%%%%%%%%%%%%%%%%%%%%%%%%%%%%%%%%%%%%%%%%%%%%%%%%%%%%%%%%%%%%%%%%%
\subsection{Quadratic Rota-Baxter Jordan algebras}

\begin{definition}
	The triple $(A, \circ, \mathcal{B}, P)$ is called a \textbf{quadratic Rota-Baxter Jordan algebra} of weight $\lambda$ if $(A, \circ, \mathcal{B})$ is a quadratic Jordan algebra and $P$ is a Rota-Baxter operator on $(A, \circ )$ of weight $\lambda$ satisfying the following compatibility condition:
	\begin{equation}
		\mathcal{B}(x, P(y)) + \mathcal{B}(P(x), y) + \lambda \mathcal{B}(x, y) = 0, \;\; \forall x, y \in A. \label{eq:qrbp}
	\end{equation}
\end{definition}

\begin{proposition}\label{prop:PandtauP}
	Let $( A, \circ, \mathcal{B})$ be a quadratic Jordan algebra and $P: A \to A$ be a linear map. Then $( A, \circ, \mathcal{B}, P)$ is a quadratic Rota-Baxter Jordan algebra of weight $\lambda$ if and only if $( A, \circ, \mathcal{B}, \tilde{P} := - \lambda \id - P)$ is a quadratic Rota-Baxter Jordan algebra of weight $\lambda$.
\end{proposition}
\begin{proof}
	Obviously, $P$ is a Rota-Baxter operator of weight $\lambda$ if and only if $\tilde{P}$ is a Rota-Baxter operator of weight $\lambda$.
	Moreover, for all $x, y \in A$, we have
	\begin{equation*}
		\mathcal{B}(\tilde{P}(x), y)+\mathcal{B}(x,\tilde{P}(y))+\lambda \mathcal{B}(x, y)=-\mathcal{B}(P(x), y)-\mathcal{B}(x, P(y))-\lambda \mathcal{B}(x, y).
	\end{equation*}
	Thus, $P$ satisfies Eq.~\eqref{eq:qrbp} if and only if $\tilde{P}$ satisfies Eq.~\eqref{eq:qrbp}.
	The proof is completed.
\end{proof}

\begin{proposition}\label{exrbfna}
	Let $(A, \circ)$ be a Jordan algebra, $P$ be a Rota-Baxter operator of weight $\lambda$ on $(A, \circ)$ and $(A \ltimes_{L_{\circ}^{*}} A^{*}, \bullet)$ be the semi-direct product of $(A, \circ)$ and $(A^*; L_{\circ}^*)$. 
	Then $(A \ltimes_{L_{\circ}^{*}} A^{*}, \bullet , \mathcal{B}_d, P+\tilde{P}^{*})$ is a quadratic Rota-Baxter Jordan algebra of weight $\lambda$, where $\mathcal{B}_d$ is defined by Eq.~\eqref{eq:abm} and $\tilde{P} := - \lambda \id - P$.
\end{proposition}
\begin{proof}
	By Proposition \ref{GouzaoQJA}, $( A \ltimes_{L_{\circ}^{*}} A^{*}, \bullet , \mathcal{B}_d)$ is a quadratic Jordan algebra. 	
	It is easy to show that $P+\tilde{P}^{*}$ is a Rota-Baxter operator of weight $\lambda$ on Jordan algebra $(A \ltimes_{L_{\circ}^{*}} A^{*}, \bullet)$. 
	For all $x, y \in A, \; a^{*}, b^{*} \in A^{*}$, we have
	\begin{align*}
		&\mathcal{B}_{d}( P(x)+\tilde{P}^*{(a^{*})}, y+b^{*}) + \mathcal{B}_{d}( x+a^{*} , P(y)+\tilde{P}^*{(b^{*})}) + \lambda\mathcal{B}_{d}(x+a^{*}, y+b^{*}) \\
		=\; & \langle b^{*}, P(x) \rangle + \langle \tilde{P}^*{(a^{*})}, y \rangle + \langle \tilde{P}^*{(b^{*})}, x \rangle + \langle a^{*}, P(y) \rangle + \lambda\langle b^{*}, x \rangle + \lambda\langle a^{*}, y \rangle \\
		=\; & \langle b^{*}, P(x) \rangle + \langle a^{*}, \tilde{P}{(y)} \rangle + \langle b^{*}, \tilde{P}{(x)} \rangle + \langle a^{*}, P(y) \rangle + \lambda\langle b^{*}, x \rangle + \lambda\langle a^{*}, y \rangle = 0.
	\end{align*}
	Thus, $P+\tilde{P}^{*}$ satisfies Eq.~\eqref{eq:qrbp} and $(A \ltimes_{L_{\circ}^{*}} A^{*}, \bullet, \mathcal{B}_d, P+\tilde{P}^{*})$ is a quadratic Rota-Baxter Jordan algebra of weight $\lambda$.
\end{proof}

\begin{example}\label{RBJordan2}
	Let $(A, \circ)$ be the 3-dimensional Jordan algebra with a basis $\{e_{1}, e_{2}, e_{3}\}$ in Example~\ref{ex:JordanXY} (b). 
	Define a linear map $P: A \to A$ by
	\begin{align*}
		P(e_1) = 0, \quad P(e_2) = e_2+ e_3, \quad P(e_3) = -2(e_2+ e_3).
	\end{align*}
	Then $P$ is a Rota-Baxter operator of weight $1$ on $(A, \circ)$.
	Let $\{e_1^*, e_2^*, e_3^*\}$ be the dual basis of $\{e_1, e_2, e_3\}$.
	Then, by Proposition~\ref{exrbfna}, $(A \ltimes_{L_{\circ}^{*}} A^{*}, \bullet, \mathcal{B}_d, P+\tilde{P}^{*})$ is a quadratic Rota-Baxter Jordan algebra of weight $1$, whose
	non-zero product of $\bullet$ is given by
	\begin{align*}
		e_{2} \bullet e_{2} &= e_{1}, \; 
		&&
		e_{3} \bullet e_{3} = -e_{1}, \;
		&&
		e_{1} \bullet e_{i} = e_{i} = e_{i} \bullet e_{1}, \;
		&& e_2 \bullet e_1^* = e_2^* = e_1^* \bullet e_2,  \\
		e_2 \bullet e_2^* &= e_1^* = e_2^* \bullet e_2, \;
		&& e_3 \bullet e_1^* = -e_3^* = e_1^* \bullet e_3, \;
		&& e_3 \bullet e_3^* = e_1^* = e_3^* \bullet e_3, \;
		&& e_1 \bullet e_i^* = e_i^* = e_i^* \bullet e_1,
	\end{align*}
	for all $i=1,2,3$, where $\mathcal{B}_d$ is defined by
	\begin{equation*}
		\mathcal{B}_d(e_i, e_j) = \mathcal{B}_d(e_i^*, e_j^*) = 0, \quad \mathcal{B}_d(e_i^*, e_j) = \mathcal{B}_d(e_i, e_j^*) = \delta_{ij}, \;\; \forall i, j=1, 2, 3,
	\end{equation*}
	and $\tilde{P}^{*}$ is given by 
	\begin{equation*}
		\tilde{P}^{*}(e_1^{*}) = - e_1^{*}, \quad \tilde{P}^{*}(e_2^{*}) = -2e_2^{*}+ 2e_3^{*}, \quad \tilde{P}^{*}(e_3^{*}) = - e_2^{*} + e_3^{*}.
	\end{equation*}
\end{example}

The following result demonstrates that a quadratic Rota-Baxter Jordan algebra of zero weight yields a triangular Jordan D-bialgebra, while a nonzero weight one yields a factorizable Jordan D-bialgebra.
\begin{theorem}\label{thm:qrbj2qtjb}
	Let $(A, \circ, \mathcal{B}, P)$ be a quadratic Rota-Baxter Jordan algebra of weight $\lambda$ and $I_\mathcal{B}: A^* \to A$ be the induced linear isomorphism by $\mathcal{B}$.
	Let $r\in A \otimes A$ be the corresponding 2-tensor form of $P I_\mathcal{B}$.
	Then $(A, \circ, \Delta_r)$ is a quasi-triangular Jordan D-bialgebra with $\Delta_{r}$ given by Eq.~\eqref{eq:jcbd}, and $r_{+} - r_{-} = -\lambda I_{\mathfrak{B}}$.
	In particular, the following statements hold.
	\begin{enumerate}[label=(\roman*)]
		\item
		    If $\lambda = 0$, then $(A, \circ, \Delta_r)$ is a triangular Jordan D-bialgebra.

		\item
		    If $\lambda \neq 0$, then $(A, \circ, \Delta_r)$ is a factorizable Jordan D-bialgebra.
	\end{enumerate}
\end{theorem}
\begin{proof}
	By Lemma~\ref{lem:rb2eo}, $P + \frac{\lambda}{2}\id$ is an extended relative Rota-Baxter operator with extension $\frac{\lambda}{2}\id$ associated to $(A; L_{\circ})$.
	Furthermore, $P + \frac{\lambda}{2}\id$ is skew-adjoint with respect to $\mathcal{B}$.
	By Corollary~\ref{coro:exoqs}, $r$ is a solution of the JYBE in $(A, \circ)$ with invariant symmetric part, as the 2-tensor form of $(P+\frac{\lambda}{2}\id)I_\mathcal{B} - \frac{\lambda}{2}I_\mathcal{B}$.
	Moreover, $\tau(r)_{+} = -P I_\mathcal{B} - \lambda I_\mathcal{B}$ and $r_{+} - r_{-} = r_{+} + (\tau(r))_{+} = -\lambda I_\mathfrak{B}$.
	Now the particular cases are obvious.
\end{proof}

\begin{remark}\label{rmk:rtr}
	The above proof establishes that $\tau(r)$ is the tensor form of $-P I_\mathcal{B} - \lambda I_\mathcal{B}$.
	By Proposition~\ref{prop:PandtauP}, $(A, \circ, \mathcal{B}, -\lambda \id-P)$ is a quadratic Rota-Baxter Jordan algebra of weight $\lambda$. 
	Applying Theorem~\ref{thm:qrbj2qtjb}, we obtain that $(A, \circ, \mathcal{B}, -\lambda \id-P)$ likewise yields a quasi-triangular Jordan D-bialgebra $(A, \circ, \Delta_{\tau(r)})$.
	In particular, if $\lambda \neq 0$, it yields a factorizable Jordan D-bialgebra $(A, \circ, \Delta_{\tau(r)})$.
\end{remark}

Complementing Theorem~\ref{thm:qrbj2qtjb}, the next result shows that every factorizable Jordan D-bialgebra naturally gives rise to a quadratic Rota-Baxter Jordan algebra of nonzero weight.
This completes the bijective correspondence between factorizable Jordan D-bialgebras and quadratic Rota-Baxter Jordan algebras of nonzero weight.

\begin{theorem}\label{thm:fpb2qrp}
	Let $(A, \circ ,\Delta_r)$ be a factorizable Jordan D-bialgebra.  
	Writing $r$ as $r = \Lambda + \Theta$ with $\Lambda \in \mathrm{Alt}^2(A)$ and $\Theta \in \mathrm{Sym}^2(A)$.
	Define a linear map $P: A \to A$ and a bilinear form $\mathcal{B}$ on $A$ respectively by
	\begin{align*}
		P & := -\frac{\lambda}{2} r_{+}  \Theta_{+}^{-1} , \quad \lambda \neq 0, \\
		\mathcal{B}(x, y) & := -\frac{\lambda}{2}\langle \Theta_{+}^{-1} (x) , y \rangle, \;\; \forall x, y \in A.
	\end{align*}
	Then $( A, \circ, \mathcal{B}, P)$ is a quadratic Rota-Baxter Jordan algebra of weight $\lambda$.
\end{theorem}
\begin{proof}
	Since $\Theta_{+} = \frac{1}{2}(r_{+} - r_{-})$ is a linear isomorphism of vector spaces, $\mathcal{B}$ is a nondegenerate bilinear form.
	Note that $\Theta_{+}^* = \Theta_{+}$, we have $(\Theta_{+}^{-1})^* = (\Theta_{+}^{*}) ^{-1} = \Theta_{+}^{-1}$ and therefore $\mathcal{B}$ is symmetric.
	Moreover, for all $x, y, z \in A$, we have
	\begin{align*}
		&\mathcal{B}(x \circ y, z)  = -\frac{\lambda}{2}\langle \Theta_+^{-1} (x \circ y) , z \rangle \overset{\eqref{eq:adi1}}{=} -\frac{\lambda}{2}\langle L_{\circ}^*(y)\Theta_+^{-1} (x) , z \rangle \\
		=\;& -\frac{\lambda}{2}\langle \Theta_+^{-1} (x) , y \circ z \rangle = -\frac{\lambda}{2}\langle x , \Theta_+^{-1}(y \circ z) \rangle = \mathcal{B}(x, y \circ z).
	\end{align*}
	Thus, $(A, \circ, \mathcal{B})$ is a quadratic Jordan algebra.
	For all $a^*, b^* \in A^*$, we have
	\begin{align*}
		& \mathcal{B}(P(\Theta_{+}(a^*)), \Theta_{+}(b^*)) + \mathcal{B}(\Theta_{+}(a^*), P(\Theta_{+}(b^*))) + \lambda \mathcal{B}(\Theta_{+}(a^*), \Theta_{+}(b^*))                              \\
		=\; & -\frac{\lambda}{2}\mathcal{B}(r_{+}(a^*), \Theta_{+}(b^*)) -\frac{\lambda}{2} \mathcal{B}(\Theta_{+}(a^*), r_{+}(b^*)) + \lambda \mathcal{B}(\Theta_{+}(a^*), \Theta_{+}(b^*)) \\
		=\; & \frac{\lambda^2}{4} \langle b^*, r_{+}(a^*)\rangle + \frac{\lambda^2}{4} \langle a^*, r_{+}(b^*)\rangle -\frac{\lambda^2}{2} \langle a^*, \Theta_{+}(b^*)\rangle   \\
		=\; & \frac{\lambda^2}{4}\langle a^* \otimes b^*, r + \tau(r) \rangle - \frac{\lambda^2}{2} \langle a^* \otimes b^*, \Theta \rangle =  0.
	\end{align*}
	That is, $P$ satisfies Eq.~\eqref{eq:qrbp} and $P + \frac{\lambda}{2}\id$ is skew-adjoint with respect to $\mathcal{B}$.
	Furthermore, since $I_\mathcal{B}^{-1} = - \frac{\lambda}{2}  \Theta_{+}^{-1}$, we have $r_{+} = PI_\mathcal{B}$.
	Then, by Corollary~\ref{coro:exoqs}, $P + \frac{\lambda}{2}\id$ is an extended relative Rota-Baxter operator with extension $\frac{\lambda}{2}\id$ associated to $(A; L_{\circ})$.
	Lemma~\ref{lem:rb2eo} then implies that $P$ is a Rota-Baxter Jordan algebra of weight $\lambda$.
	This completes the proof.
\end{proof}

\begin{proposition}
	Let $(A, \circ, \Delta_r)$ be a factorizable Jordan D-bialgebra, which corresponds to a quadratic Rota-Baxter Jordan algebra $(A, \circ, \mathcal{B}, P)$ of weight $\lambda \neq 0$ via Theorems~\ref{thm:fpb2qrp} and \ref{thm:qrbj2qtjb}.
	Then the factorizable Jordan D-bialgebra $(A, \circ, \Delta_{\tau{(r)}})$ corresponds to the quadratic Rota-Baxter Jordan algebra $(A, \circ, \mathcal{B}, - \lambda \id - P)$ of weight $\lambda$.
	In conclusion, we have the following commutative diagram.
	\begin{displaymath}
		\xymatrix{
			(A, \circ, \Delta_r) \ar@{<->}[rr]^{\mathrm{Prop.}~\ref{prop:factorJba}}\ar@{->}[d]^{\mathrm{Thm.}~\ref{thm:fpb2qrp}}& & (A, \circ, \Delta_{\tau(r)}) \ar@{->}[d]^{\mathrm{Thm.}~\ref{thm:fpb2qrp}} \\
			(A, \circ, \mathcal{B}, P) \ar@<1.5ex>[u]^{\mathrm{Thm.}~\ref{thm:qrbj2qtjb}} \ar@{<->}[rr]^{\mathrm{Prop}.~\ref{prop:PandtauP} }& & (A, \circ, \mathcal{B}, -\lambda \id -P) \ar@<1.5ex>[u]^{\mathrm{Thm.}~\ref{thm:qrbj2qtjb}}}
	\end{displaymath}
\end{proposition}
\begin{proof}
	Let $\Theta: = \frac{1}{2}(r + \tau(r))$ be the symmetric part of $r$. 
	By Theorem~\ref{thm:fpb2qrp}, the factorizable Jordan bialgebra $(A, \circ, \Delta_{\tau(r)})$ induces a quadratic Rota-Baxter Jordan algebra $(A, \circ, \mathcal{B}^{\prime}, P^{\prime})$ of weight $\lambda \ne 0$, where
	\begin{align*}
		\mathcal{B}^{\prime}(x, y) & =-\frac{\lambda}{2}\langle \Theta_{+}^{-1} (x) , y \rangle= \mathcal{B}(x, y), \quad \forall x, y \in A, \\
		P^{\prime} & = -\frac{\lambda}{2} (\tau(r))_{+}  \Theta_{+}^{-1} = \frac{\lambda}{2} r_{-}  \Theta_{+}^{-1} = \frac{\lambda}{2} (r_{+}  - 2 \Theta_{+})  \Theta_{+}^{-1} = -P -\lambda \mathrm{id}.
	\end{align*}
	That is, $(A,\circ, \Delta_{\tau(r)})$ induces a quadratic Rota-Baxter Jordan algebra $(A, \circ, \mathcal{B}, -P -\lambda \mathrm{id})$ of weight $\lambda$ by Theorem~\ref{thm:fpb2qrp}.
	Conversely, Remark~\ref{rmk:rtr} implies that the quadratic Rota-Baxter Jordan algebra $(A, \circ, \mathcal{B}, -\lambda \mathrm{id}-P)$ induces the factorizable Jordan bialgebra $(A, \circ, \Delta_{\tau(r)})$ via Theorem~\ref{thm:qrbj2qtjb}.
\end{proof}

\begin{proposition}\label{RBJ:XY}
	Let $(A, \circ)$ be a Jordan algebra, $P$ be a Rota-Baxter operator of weight $\lambda \ne 0$ on $(A, \circ)$ and $(A \ltimes_{L_{\circ}^{*}} A^{*}, \bullet)$ be the semi-direct product of $(A, \circ)$ and $(A^*; L_{\circ}^*)$. 
	Let $\{e_1, e_2, \cdots, e_n\}$ be a basis of $A$ and $\{e_1^*, e_2^*, \cdots, e_n^*\}$ be the dual basis of $A^*$. Define $\tilde{P} := - \lambda \id - P$ and set
	\begin{equation*}
		r := \sum_{i}^n \tilde{P} (e_i) \otimes e_i^{*} + e_i^*\otimes P(e_i) \in (A \oplus A^*) \otimes (A \oplus A^*).
	\end{equation*}
	Then $(A \ltimes_{L_{\circ}^{*}} A^{*}, \bullet, \Delta_r)$ is a factorizable Jordan D-bialgebra, where $\Delta_{r}$ is defined by Eq.~\eqref{eq:jcbd}.
\end{proposition}
\begin{proof}
	By Proposition~\ref{exrbfna}, $(A \ltimes_{L_{\circ}^{*}} A^{*}, \bullet, \mathcal{B}_d, P + \tilde{P}^{*})$ is a quadratic Rota-Baxter Jordan algebra of weight $\lambda$ with $\mathcal{B}_d$ defined by Eq.~\eqref{eq:abm}. 
	For all $x \in A$ and $a^* \in A^*$, we have
	\begin{align*}
		r_{+}(a^* + x) =  \sum_{i}\langle x, e_{i}^*\rangle P(e_i) + \sum_{i}\langle a^*, \tilde{P}(e_i)\rangle e_{i}^* = P(x) + \tilde{P}^{*}(a^*) =  (P+\tilde{P}^{*}) I_{\mathcal{B}_d}(a^* + x).
	\end{align*}
	That is, $r$ is the 2-tensor form of $(P+\tilde{P}^{*}) I_{\mathcal{B}_d}$.
	By Theorem \ref{thm:qrbj2qtjb}, $(A \ltimes_{L_{\circ}^{*}} A^{*}, \bullet, \Delta_r)$ is a factorizable Jordan D-bialgebra.
\end{proof}

\begin{corollary}\label{RBPoissonof3}
	Let $(A, \circ)$ be a Jordan algebra and $(A \ltimes_{L_{\circ}^{*}} A^{*}, \bullet)$ be the semi-direct product of $(A, \circ)$ and $(A^*; L_{\circ}^*)$.
	Let $\{e_1, e_2, \cdots, e_n\}$ be a basis of $A$ and $\{e_1^*, e_2^*, \cdots, e_n^*\}$ be the dual basis of $A^*$. 
	Set
	\begin{align*}
		r := \sum_{i}^n e_i^*\otimes e_i.
	\end{align*}
	Then $(A \ltimes_{L_{\circ}^{*}} A^{*}, \bullet, \Delta_r)$ is a factorizable Jordan D-bialgebra, where $\Delta_{r}$ is defined by Eq.~\eqref{eq:jcbd}.
\end{corollary}
\begin{proof}
	Clearly, $\id: A \to A$ is a Rota-Baxter operator of weight $-1$ on the Jordan algebra $(A, \circ)$. 
	The conclusion then follows from Proposition~\ref{RBJ:XY}.
\end{proof}

This section ends with an explicit example that illustrates the construction of factorizable Jordan D-bialgebras from finite-dimensional Jordan algebras provided by Corollary~\ref{RBPoissonof3}.

\begin{example}\label{FJordanDB1}
	Continuing with the notations in Example \ref{ex:Jordan3}, 	
	there is a factorizable Jordan D-bialgebra $(A \ltimes_{L_{\circ}^{*}} A^{*}, \bullet, \Delta_r)$ with $\Delta_{r}$ explicitly given by
	\begin{equation*}
		\Delta_r(e_1)  = 0, \quad
		\Delta_r(e_2) = \frac{1}{2} e_1^* \otimes e_2, \quad
		\Delta_r(e_1^*) = e_1^* \otimes e_1^*, \quad 
		\Delta_r(e_2^*) = \frac{1}{2} e_1^* \otimes e_2^*.
	\end{equation*}
\end{example}

\bigskip

\noindent \textbf{Data Availability Statement} No datasets were generated or analyzed during the current study.

%%%%%%%%%%%%%%%%%%%%%%%%%%%%%%%%%%%%%%%%%%%%%%%%%%%%%%%%%%%%%%%%%%%%%%%%%%%%%%%%

%\bibliographystyle{camsplain} % latex makebst
%\bibliography{ref.bib}

\end{document}